\documentclass[12pt,a4paper, reqno]{amsart}
\usepackage[text={460pt, 620pt},centering]{geometry}
\usepackage{bbm}
\usepackage{graphics}
\usepackage{url}
\usepackage[utf8]{inputenc}
\usepackage{chngcntr}
\usepackage{lipsum}
\usepackage{amsthm}
\usepackage{amsfonts}
	\usepackage{cases}
\usepackage{color}
\usepackage{geometry}
\usepackage{amsmath}
\usepackage{amssymb}
\usepackage{latexsym}
\usepackage{amssymb}
\usepackage{mathrsfs}
\usepackage{amsmath}
\usepackage{enumerate}
\usepackage{hyperref}
\hypersetup{
    colorlinks=true,                          
    linkcolor=blue, 
    citecolor=red, 
    urlcolor=blue  } 

\usepackage{eurosym}
	\usepackage{cases}
\usepackage{color}

\newcommand{\kom}[1]{}
\renewcommand{\kom}[1]{{\bf [#1]}}

 \def\1{\raisebox{2pt}{\rm{$\chi$}}}

\newtheorem{theorem}{Theorem}[section]

\newtheorem{lemma}[theorem]{Lemma}

\newtheorem{definition}[theorem]{Definition}

\newtheorem{remark}[theorem]{Remark}

\numberwithin{equation}{section}

\newcommand{\R}{{\mathbb R}}

\newcommand{\N}{{\mathbb N}}

\newcommand{\C}{{\mathcal C}}

 \newcommand{\eps}{{\varepsilon}}
 \def\1{\raisebox{2pt}{\rm{$\chi$}}}
 
\newcommand{\abs}[1]{\left|#1\right|}
\newcommand{\norm}[1]{\left|\left|#1\right|\right|}

\newcommand{\osc}{\operatorname{osc}}

\def\vint_#1{\mathchoice%
          {\mathop{\kern 0.2em\vrule width 0.6em height 0.69678ex depth -0.58065ex
                  \kern -0.8em \intop}\nolimits_{\kern -0.4em#1}}%
          {\mathop{\kern 0.1em\vrule width 0.5em height 0.69678ex depth -0.60387ex
                  \kern -0.6em \intop}\nolimits_{#1}}%
          {\mathop{\kern 0.1em\vrule width 0.5em height 0.69678ex
              depth -0.60387ex
                  \kern -0.6em \intop}\nolimits_{#1}}%
          {\mathop{\kern 0.1em\vrule width 0.5em height 0.69678ex depth -0.60387ex
                  \kern -0.6em \intop}\nolimits_{#1}}}
\def\vintslides_#1{\mathchoice%
          {\mathop{\kern 0.1em\vrule width 0.5em height 0.697ex depth -0.581ex
                  \kern -0.6em \intop}\nolimits_{\kern -0.4em#1}}%
          {\mathop{\kern 0.1em\vrule width 0.3em height 0.697ex depth -0.604ex
                  \kern -0.4em \intop}\nolimits_{#1}}%
          {\mathop{\kern 0.1em\vrule width 0.3em height 0.697ex depth -0.604ex
                  \kern -0.4em \intop}\nolimits_{#1}}%
          {\mathop{\kern 0.1em\vrule width 0.3em height 0.697ex depth -0.604ex
                  \kern -0.4em \intop}\nolimits_{#1}}}

\newcommand{\aveint}[2]{\mathchoice%
          {\mathop{\kern 0.2em\vrule width 0.6em height 0.69678ex depth -0.58065ex
                  \kern -0.8em \intop}\nolimits_{\kern -0.45em#1}^{#2}}%
          {\mathop{\kern 0.1em\vrule width 0.5em height 0.69678ex depth -0.60387ex
                  \kern -0.6em \intop}\nolimits_{#1}^{#2}}%
          {\mathop{\kern 0.1em\vrule width 0.5em height 0.69678ex depth -0.60387ex
                  \kern -0.6em \intop}\nolimits_{#1}^{#2}}%
          {\mathop{\kern 0.1em\vrule width 0.5em height 0.69678ex depth -0.60387ex
                  \kern -0.6em \intop}\nolimits_{#1}^{#2}}}

\newcommand{\ol}{\overline}

\newcommand{\dist}{\operatorname{dist}}

\newcommand{\vp}{\varphi}

\newcommand{\tr}{\operatorname{tr}}

\newcommand{\A}{\mathcal A}
\begin{document}
\author[A. Attouchi]{Amal Attouchi}
\address[A. Attouchi]{Department of Mathematics and Statistics, University of Jyväskylä, PO Box 35, 
FI-40014 Jyväskylä, Finland }
\email{amal.a.attouchi@jyu.fi}
\title{Local regularity for  quasi-linear parabolic  equations in non-divergence form}
\date{\today}
\begin{abstract}
We consider viscosity solutions to non-homogeneous degenerate and  singular parabolic
equations of the $p$-Laplacian type and in non-divergence form.  We provide local Hölder and Lipschitz estimates for the solutions.   In the degenerate case, we prove the Hölder regularity of the gradient. Our study is based on a combination of the method of alternatives and the improvement of flatness estimates.
\end{abstract}
\maketitle
\tableofcontents
\section{Introduction}
\label{sec:intro}

We are interested in the  regularity of  viscosity solutions of the following degenerate or singular parabolic equation in non-divergence form:
\begin{equation}\label{maineq}
	\partial_t u-|Du|^\gamma\left[\Delta u+(p-2) \left\langle D^2u\frac{Du}{\abs{Du}}, \frac{Du}{\abs{Du}}\right\rangle\right]=f\quad \text{in}\quad Q_1,
\end{equation}
where $-1<\gamma<\infty$, $1<p<\infty$ and $f$ is a continuous and bounded function.
Existence and uniqueness of solutions to \eqref{maineq}  were proved in \cite{dem11}, where   more general singular  or degenerate parabolic equations were considered (see also \cite{bata, OS} and the references therein). In \cite{dem11}, Demengel established global Hölder regularity results for the solutions of the Cauchy-Dirichlet problem associated to \eqref{maineq}, under
the assumptions that $f$
is continuous and bounded in space and Hölder in time, and that the boundary data is Hölderian in space and Lipschitz in time.\\

In this work,  we investigate the higher regularity of the solution $u$ to \eqref{maineq}.
We focus on interior regularity for the gradient, away from boundaries. Let us mention two special cases. The case $\gamma=0$ corresponds to the normalized $p$-Laplacian 
$$\Delta_p^Nu:=\Delta u+(p-2) \left\langle D^2u\frac{Du}{\abs{Du}}, \frac{Du}{\abs{Du}}\right\rangle,$$ 
 and the regularity of the gradient was studied in \cite{AP17,jinsl15} using viscosity theory methods. The case $\gamma=p-2$  corresponds to the usual parabolic $p$-Laplace equations, and it was shown in \cite{julm} that bounded weak solutions and viscosity solutions are equivalent. From this equivalence, we get the Hölder regularity of the gradient for bounded $f$ using variational methods \cite{diben, dibfrie,kuusiM12,W86}. Let us also mention that recently Parviainen and  V\'azquez \cite{PV18} established an equivalence between the radial solutions of \eqref{maineq} and the radial  solutions of the standard parabolic $\gamma+2$-Laplace equation posed in a {\em fictitious} dimension. Notice that the regularity theory for \eqref{maineq} does not fall into the classical framework of fully nonlinear uniformly parabolic equations studied in \cite{wang1, wang2} due to the lack of uniform ellipticity and the presence of singularities.\\
 
In this paper, we provide local Lipschitz estimates for solutions to \eqref{maineq} in the whole range $-1<\gamma<\infty$. For  $\gamma>0$, we prove the Hölder regularity  of the gradient. 
Recently, for $\gamma\neq 0$, the homogeneous case $f=0$ was treated by Imbert, Jin and Silvestre \cite{IJS17}.  The case where $f$ depends only on $t$ can be handled using the results of  \cite{IJS17}, since $\tilde u(x,t): =u(x, t)-\int_0^t f(s)\,ds$ solves the homogeneous equation. If we assume more regularity on $f$,  let us  say $f\in C^{1,0}_{x,t}(Q_1)$, then one could adapt the argument of \cite{IJS17}  by regularizing the equation and differentiating it, and  then prove the Hölder continuity of the gradient of the solutions of \eqref{maineq} with a norm which will then depend on $\norm{Df}_{L^\infty(Q_1)}$.  Our study relies on a nonlinear method based on  compactness arguments
where we avoid differentiating the equation and assume only  the continuity of  $f$. There are different characterizations of pointwise $C^{1+\alpha, \frac{1+\alpha}{2}}$ functions, and we will use the one relying on the rate of approximations by planes.
 The study is based on estimates which prove that the solution gets
flatter and flatter, when zooming into the smaller scales. There are three  key points: an improvement of flatness estimate,  the method of alternatives and the intrinsic scaling technique. 
In  the degenerate case  $\gamma>0$,   in order to prove the Hölder regularity of the gradient, one has to choose a suitable scaling that takes into account the structure of the equation. 
Indeed, when the equation degenerates, the solutions locally  generate their own scaling (''intrinsic scaling'') according to the values of the diffusion coefficients.   The main idea behind the  intrinsic scaling technique  is  to study the equation  not on all parabolic cylinders, but rather on
those whose ratio between space and time lengths depend on the size
of the solution itself on the same cylinder, according to the regularity considered \cite{diben, urbano}.  Specifically,    in order to prove Hölder regularity  of the gradient, we   consider the so called
intrinsic cylinders defined by
$$Q_r^\lambda(x_0, t_0):=B(x_0, r)\times (t_0-\lambda^{-\gamma } r^2, t_0],$$
 where the parameter  $\lambda>0$  behaves  like $\underset{Q_r^\lambda}{\sup} |Du|\approx  \lambda$  (see Sections \ref{sect4} and  \ref{sect5}). \\

Our strategy is to combine an improvement of flatness method with the method of alternatives (the Degenerate  Alternative and the Smooth Alternative). This procedure defines  an iteration that stops in the case where we reach a cylinder where the Smooth Alternative  holds. More precisely, using an iteration process and compactness arguments, our aim is to prove  that 
there exist $\rho=\rho(p,n,\gamma)>0$ and $\delta=\delta(p,n, \gamma) \in (0,1)$ with $\rho<(1-\delta)^{\gamma+1}$ such that  one of the two following alternative holds:\\

\begin{itemize}
\item{\bf Degenerate Alternative:} For every $k\in \N$ there exists a vector $l_k$ {\em with} $|l_k|\leq C(1-\delta)^k$ such that 
$$\underset{(x,t)\in Q_{r_k}^{\lambda_k}}{\osc}\, (u(x,t)-l_k\cdot x)\leq r_k\lambda_k,$$
where $r_k:=\rho^k$, $\lambda_k:=(1-\delta)^k$ and $Q_{r_k}^{\lambda_k}:=B_{r_k}(0)\times (-r_k^2\lambda_k^{-\gamma}, 0]$. That is, we have an improvement of flatness at all scales.

\item {\bf Smooth Alternative:} The previous iteration stops at some step $k_0 $, that is,  $|l_{k_0}|\geq C(1-\delta)^{k_0}$, and we can show that the gradient of $u$ stays away from 0 in some cylinder and then use the known results for uniformly parabolic equations with smooth coefficients \cite{LU68, Lieb96}.
\end{itemize}

Notice that the intrinsic scaling plays a role in the choice of the cylinders $Q_{r_k}^{\lambda_k}$ in order to proceed with the iteration, and that if $|l_k|\leq C(1-\delta)^k$ for all $k$, then $|Du(0,0)|=0$. Let us explain how these alternatives appear. The existence of the vector $l_{k+1}$ in the iteration process can be reduced to the proof of an improvement of flatness (see Section \ref{sect4}) for the function

$$w_k(x,t)=\frac{u(r_k x, r_k^{2} \lambda_k^{-\gamma} t)-l_k\cdot r_k x}{r_k\lambda_k}.$$
The function $w_k$ solves 
\begin{equation*}\label{introdev}\partial_t w_k-\left|Dw_k+\frac{l_k}{\lambda_k}\right|^\gamma\left[\Delta w_k+(p-2) \left\langle D^2w_k\frac{Dw_k+l_k/\lambda_k}{\abs{Dw_k+l_k/\lambda_k}}, \frac{Dw_k+l_k/\lambda_k}{\abs{Dw_k+l_k/\lambda_k}}\right\rangle\right]=\bar f\quad\text{in}\,\,Q_1,
\end{equation*}
where $\bar f(x,t):=r_k\lambda_k^{-(\gamma+1)}f(r_kx, r_k^2\lambda_k^{-\gamma}t)$.
 This leads us to study the equation satisfied by the deviations of $u$ from planes 
 $w(x,t)=u(x,t)-q\cdot x,$

	\begin{equation}\label{devia}
	\partial_t w-|Dw+q|^\gamma\left[\Delta w+(p-2) \left\langle D^2w\frac{Dw+q}{\abs{Dw+q}}, \frac{Dw+q}{\abs{Dw+q}}\right\rangle\right]=\bar f\quad\text{in}\quad Q_1.
	\end{equation}
	We see that $w_k$ satisfies \eqref{devia} with $q=l_k/\lambda_k$.
The proof of the improvement of flatness is based on compactness estimates  and a contradiction argument. Unlike the case of the normalized $p$-Laplacian, the ellipticity coefficients of  the equation \eqref{devia} depend on $q$. To tackle this problem, we have to introduce the two alternatives: either we have a uniform bound on $|q|$ and we can run again our iteration, or $|q|$ is larger than some fixed constant. In this later case, using {\bf Lipschitz estimates} in the space  variable which are {\bf independent of} $q$ (see Lemma \ref{Liphom1}), we can provide a strictly positive lower bound  for the gradient  of $u$ and finish the proof by using known results for uniformly parabolic equations with smooth coefficients. Our main result is the following.

 \begin{theorem}\label{mainth}
 Let $0\leq\gamma<\infty$ and $1<p<\infty$. Assume that $f$ is a continuous and bounded  function, and let $u$ be a bounded viscosity solution of \eqref{maineq}. Then $u$ has a locally Hölder continuous gradient, and
  there exist a constant $\alpha=\alpha(p,n, \gamma)$ with $\alpha\in (0,\frac{1}{1+\gamma})$ and a constant $C=C(p,n, \gamma)>0$ such that 
 \begin{equation}
 |D u(x,t)-D u(y,s)|\leq C\left(1+\norm{u}_{L^\infty(Q_1)}+\norm{f}_{L^\infty(Q_1)}\right)(|x-y|^{\alpha}+|t-s|^{\frac\alpha2})
 \end{equation}
 and 
 \begin{equation}
 |u(x,t)-u(x,s)|\leq C \left(1+\norm{u}_{L^\infty(Q_1)}+\norm{f}_{L^\infty(Q_1)}\right)|s-t|^{\frac{1+\alpha}{2}}.
 \end{equation}
 \end{theorem}
The method of improvement of flatness  was already used in the elliptic case \cite{ARP,BD,IS} and in the uniformly parabolic case \cite{AP17}. In these works, one ends up working with equations with ellipticity constants not depending on the slope, making the improvement of flatness  working for all $k\in \N$. The method of alternatives is classical when studying the regularity for $p$-Laplacian type equations \cite{diben,dibfrie,kuusiM12}. In the singular case $-1<\gamma<0$, we weren't able to provide uniform (with respect to $q$) Lipschitz estimates for solutions to \eqref{devia}. The higher regularity of the gradient is still an open problem  when $\gamma<0$.\\

The paper is organized as follows. In Section \ref{sect2} we fix the notations, gather some known regularity results that we will use later on,  and reduce the problem by re-scaling. Section \ref{sect3} is devoted to the study of the equation  \eqref{devia} and provides the needed compactness estimates. In Section \ref{sect4},  we provide the proof of the “improvement of flatness” property.   
In Section \ref{sect5} we prove Theorem \ref{mainth} proceeding by iteration and considering the two possible alternatives.  Section \ref{sect6}  contains the proof of the Lipschitz regularity for solutions to \eqref{maineq} for $-1<\gamma<\infty$.  In Section \ref{sect7}  we prove the uniform Lipschitz estimates for solutions to \eqref{devia} for $0\leq \gamma<\infty$.\\

\noindent{\bf Acknowledgment.} The author  is supported by the Academy of Finland, project
number 307870.
 \section{Preliminaries and notations}\label{sect2}
 In this section we fix the notation that we are going to use throughout the paper, recall the definitions of parabolic Hölder spaces and precise the definition of viscosity solutions that we adopt.\\
 
 \noindent{\bf Notations.} 
 For $x_0\in\R^n$, $t_0\in \R$ and $r>0$ we denote the Euclidean ball
 $$B_r(x_0)=B(x_0, r):= \left\{ x\in \R^n\quad|\quad |x-x_0|<r\right\},$$
 and the parabolic cylinder
 $$Q_r(x_0, t_0):=B_r(x_0)\times (t_0-r^2, t_0].$$
 We also define the re-scaled (or intrinsic) parabolic cylinders
 $$Q_r^{\lambda}(x_0, t_0):= B_r(x_0)\times (t_0-r^2\lambda^{-\gamma}, t_0],$$
 which are suitably scaled to reflect the degeneracy   of the equation \eqref{maineq}.
 When $x_0=0$ and $t_0=0$ we omit indicating the centers in the above notations.  For a set $E\subset \R^{n+1}$, $\partial_p E$ is its parabolic boundary. For the parabolic functional classes, we use the following notations.
 For $\alpha\in (0,1)$,  we use the notation
	\[
	[u]_{C^{\alpha, \alpha/2}(Q_r)}:=\sup_{\substack{(x,t),(y,s)\in Q_r,\\
			(x,t)\neq (y,s)}} \dfrac{|u(y,s)-u(x,t)|}{|x-y|^\alpha+|t-s|^{\frac\alpha 2}},
	\]
	\[
	\norm{u}_{C^{\alpha, \alpha/2}(Q_r)}:=\norm{u}_{L^\infty(Q_r)}+[u]_{C^{\alpha, \alpha/2}(Q_r)}
	\]
	for H\"older continuous functions.
The space $C^{1+\alpha, (1+\alpha)/2}(Q_r)$ is defined  as the space of all functions with a finite norm 
\[
	\norm{u}_{C^{1+\alpha, (1+\alpha)/2}(Q_r)}:=\norm{u}_{L^\infty(Q_r)}+\norm{Du}_{L^\infty(Q_r)}+[u]_{C^{1+\alpha, (1+\alpha)/2}(Q_r)},
	\]
	where
		\begin{align*}
[u]_{C^{1+\alpha,(1+\alpha)/2}(Q_r)}:&=\sup_{\substack{(x,t), (y,s)\in Q_r\\(x,t)\neq (y,s)}}\dfrac{|D u(x,t)-D u(y, s)|}{|x-y|^\alpha+|t-s|^{\frac\alpha 2}}\\
	&\quad+\sup_{\substack{(x,t),(x,s)\in Q_r,\\ t\neq s}}\,\frac{|u(x,t)-u(x,s)|}{|t-s|^{\frac{1+\alpha}{2}}}.
\end{align*}
 In this paper $C$ will denote generic constants which may change
from line to line. If a
more careful control over a constant is needed, we denote its dependence on certain parameters writing
$C$(parameters).\\

\noindent\textbf{Definition of solutions}.  We adopt the same notion of viscosity solutions to \eqref{maineq} as  the one used in \cite{dem11, IJS17}.
For the existence and uniqueness of the solutions and the comparison principles for equations of type   \eqref{maineq}, we refer the reader to \cite{chen, dem11, giga, IJS17, OS}.  For the very singular case $\gamma<0$, the definition in the sense of Ohnuma-Sato \cite{OS}  requires the introductio of   a set of admissible test functions when the gradient  of $u$ is 0 whereas no special restrictions are needed in the degenerate case $\gamma>0$. However, for $\gamma\neq 0$ the notion of solutions in the sense of Ohnuma-Sato and the one proposed by Demengel are equivalent (see Appenidx of \cite{dem11}). One can also show that the notion of solutions in \cite{dem11} is equivalent to the one proposed by \cite{julm} (this was done in  \cite{ARP} for the elliptic case and it can be easily generalized to the parabolic case). In the proofs of the Hölder estimates in time we will rely on those comparison principles \cite[Theorem 1]{dem11}. We will also use the stability results for \eqref{maineq} (see \cite[Proposition 3]{dem11} and \cite[Theorem 6.1, 6.2]{OS}). Let us recall the definition of viscosity solutions \cite{dem11}.
\begin{definition} A locally bounded and  upper semi-continuous function $u$ in $Q_1$ is called a viscosity
subsolution  of \eqref{maineq}  if, for any  point $(x_0, t_0)\in Q_1$, one of the following conditions holds
\begin{enumerate}[i)]
\item 
Either  for every  $\varphi\in C^2(Q_1)$, such that  $u -\varphi$  has a local maximum at ($x_0, t_0)$ and $D\varphi(x_0, t_0)\neq 0$ it holds
$$\partial_t\varphi(x_0, t_0)-|D\varphi(x_0, t_o)|^\gamma \Delta_p^N\varphi(x_0,t_0)\leq f(x_0,t_0). $$
\item Or if there exists $\delta_1$ and $\varphi\in C^2(]t_0-\delta_1, t_0+\delta_1[)$, such that
\begin{equation}
\begin{cases}
\varphi(t_0)=0\\
u(x_0, t_0)\geq u(x_0, t)-\varphi(t) \quad\text{for all}\quad t\in ]t_0-\delta_1, t_0+\delta_1[\\
\underset{]t_0-\delta_1, t_0+\delta_1[}{\sup}\, (u(x,t)-\varphi(t))\quad\text{is constant in a neighborhood of}\quad x_0,
\end{cases}
\end{equation}
then $$\varphi'(t_0)\leq f(x_0, t_0).$$
\end{enumerate}

A locally bounded and  lower semi-continuous function $u$ in $Q_1$ is called a viscosity
supersolution  of \eqref{maineq}  if, for any  point $(x_0, t_0)\in Q_1$ one of the following conditions holds
\begin{enumerate}[i)]
\item 
Either  for every  $\varphi\in C^2(Q_1)$, such that  $u -\varphi$  has a local minimum at ($x_0, t_0)$ and $D\varphi(x_0, t_0)\neq 0$ it holds
$$\partial_t\varphi(x_0, t_0)-|D\varphi(x_0, t_o)|^\gamma \Delta_p^N\varphi(x_0,t_0)\geq f(x_0,t_0). $$
\item Or if there exists $\delta_1$ and $\varphi\in C^2(]t_0-\delta_1, t_0+\delta_1[)$, such that
\begin{equation}
\begin{cases}
\varphi(t_0)=0\\
u(x_0, t_0)\leq u(x_0, t)-\varphi(t) \quad\text{for all}\quad t\in ]t_0-\delta_1, t_0+\delta_1[\\
\underset{]t_0-\delta_1, t_0+\delta_1[}{\inf}\, (u(x,t)-\varphi(t))\quad\text{is constant in a neighborhood of}\quad x_0,
\end{cases}
\end{equation}
then $$\varphi'(t_0)\geq f(x_0, t_0).$$
\end{enumerate}

A  continuous function $u$ is called a viscosity solution of \eqref{maineq}, if it is both a viscosity subsolution
and a viscosity supersolution.
\end{definition}

\noindent\textbf{Normalization and scaling}: Without a loss of generality, we may assume  in Theorem \ref{mainth}
that $\norm{u}_{L^\infty(Q_1)}\leq 1/2$  and that 
$\norm{f}_{L^\infty(Q_1)}\leq \eps_0$, where  $\eps_0=\eps_0(p,n, \gamma)>0$ will be chosen in Section \ref{sect4}.  Indeed, we can use a nonlinear method  to realize this (notice that contrary to the elliptic case, multiplying solutions by a
constant does not yield a solution to a similar equation). For  $\gamma\geq 0$, set $$\theta:= \left( 2\norm{u}_{L^\infty(Q_1)}+\left(\frac{\norm{f}_{L^\infty(Q_1)}}{\eps_0}\right)^{\frac{1}{\gamma+1}}+1\right)^{-1}.$$ 
We may consider  in $Q_1$ the function 
$$u_{\theta}(x,t):= \theta u( x, \theta^{\gamma} t).$$
The function $u_\theta$ satisfies $\norm{u_\theta}_{L^\infty(Q_1)}\leq 1/2$ and  solves in $Q_1$ 
$$\partial_tu_{\theta}=|Du_\theta|^\gamma\Delta_p^N u_\theta+ f_\theta$$
with  $$f_{\theta}(x,t):=\theta^{\gamma+1}f(x, \theta^{\gamma}t), \qquad \qquad\norm{f_\theta}_ {L^\infty(Q_1)}\leq \eps_0.$$

We will use the standard alternative characterization  (see \cite[Lemma 12.12]{Lieb96}) of functions with Hölder continuous gradient. This one  is more  suitable when we prove regularity results for nonlinear equations using compactness methods. Indeed, we will show that $u$ can be approximated by planes with a good control on the rate of approximations. By removing a constant, we may assume that $u(0,0)=0$. \\

\noindent{\bf Known regularity results.} Here we gather some known regularity results that we will need later on.  We start with the following result of \cite{IJS17}.
\begin{theorem}\label{imbert}
	Let $-1<\gamma<\infty$ and $1<p<\infty$.	Assume that $f\equiv0$ and let $w$ be a viscosity solution to equation \eqref{maineq} in $Q_1$. For all $r\in (0,\frac34)$, there exist constants $ C_0=C_0(p,n, \gamma)>0$  and $\beta_1=\beta_1(p,n, \gamma)>0$ such that 
		\begin{equation}
		\begin{split}
		\norm{w}_{C ^{1+\beta_1, (1+\beta_1)/2}( Q_{r})}\leq  C_0(1+\norm{w}_{L^\infty(Q_1)}).
		\end{split}
		\end{equation}\end{theorem}
		We will need the following result for uniformly parabolic equations with smooth coefficients depending on the gradient \cite[Theorem 1.1]{LU68} (see also \cite[Lemma 12.13]{Lieb96}).
		\begin{theorem}\label{ladyhope}
	Define $\Omega_T:=\Omega\times (-T, 0)$, where $\Omega\subset \R^n$ is a bounded domain, and let $g\in C(\Omega_T)\cap L^\infty(\Omega_T)$.	Let $w$ be a strong solution to the  equation 
		$$\partial_t w-\sum_{i,m} a_{im}(Dw)\dfrac{\partial^2 w}{\partial x_i\partial x_m}=g,  $$
			where the coefficients $a_{im}(z)$ 
		are differentiable with respect to $z$  in the set $$\left\{ (x,t)\in \Omega_T, |w(x,t)|\leq K, |Dw(x,t)|\leq K\right\}$$
		and   satisfy the following conditions:
		\begin{align*}
		i)\quad& \lambda |\xi|^2\leq \sum_{i,m} a_{im}(Dw)\xi_i\xi_m\leq \Lambda |\xi|^2\qquad \text{for some}\quad 0<\lambda\leq \Lambda,\\
	ii)\quad& \qquad\qquad	\underset{\Omega_T}{\max}\left|\dfrac{\partial a_{im}(Dw)}{\partial z_k}\right|\leq \mu_1.
		\end{align*}
	Assume also that $\norm{g}_{L^\infty(\Omega_T)}\leq \mu_2$.
Then there exists a constant 
 $\bar\alpha:=\bar\alpha(\lambda, \Lambda, K, \mu_1, \mu_2)>0$  such that  for any $Q' \subset\subset\Omega_T$  it holds
\begin{equation}
[Dw]_{C^{\bar\alpha, \bar\alpha/2}(Q')}\leq \bar C_0.
\end{equation}
where $\bar C_0:=\bar C_0(\lambda, \Lambda, 	K, \mu_1, \norm{g}_{L^\infty(\Omega_T)}, \dist(Q', \partial_p \Omega _T))>0$.
		\end{theorem}
\begin{remark}
We will apply the result of Theorem \ref{ladyhope} in Section \ref{sect5}. The result of Theorem \ref{ladyhope} was stated for strong solutions. A priori our solutions are only viscosity solutions. This can be re-mediated by approximating $g$ with smooth functions $g_\eps$. Then the corresponding solutions $w_\eps$ solve a uniformly parabolic equation with Hölder continuous coefficients. By standard regularity results (see \cite[Theorem 14.10]{Lieb96} and \cite[Theorem 5.1]{Lieb96}), we conclude that $w_\eps$ are strong solutions and they converge locally uniformly towards $w$. The most important thing is that the $C^{1+\bar\alpha, \frac{1+\bar\alpha}{2}}$ norm of $w_\eps$ does not depend on the regularity of $g_\eps$.

\end{remark}
\section{ Lipschitz estimates and study of the equation for deviation from planes}\label{sect3}
In this section we analyze the problem \eqref{devia} and provide  regularity estimates which will be needed in the next section.
These regularity results are obtained by using standard techniques in the theory of viscosity
solutions.  We first  provide local Hölder and Lipschitz estimates with respect to the space variable for viscosity solutions of \eqref{maineq}. \\

\noindent{ \bf Lipschitz and Hölder estimates for solutions to \eqref{maineq}}
 Let us recall that, in the setting of viscosity solutions, there are essentially two approaches for proving Hölder regularity: either by Aleksandrov-Bakelman–Pucci (ABP)  estimates,  Krylov-Safonov  estimates \cite{ks79, ks80} and Harnack type inequalities, or by the Ishii-Lions’s method  \cite{ishiilions}. The first method is suitable for uniformly parabolic equations (for example for $\gamma=0$).  We point out that Argiolas-Charro-Peral \cite{arg} provided an ABP estimate for solutions of \eqref{maineq} and that recently for $f=0$, Parviainen and Vázquez \cite{PV18} provided  a Harnack estimate for solutions of \eqref{maineq}. 
More  direct viscosity methods like the method proposed by Ishii and Lions apply under weaker ellipticity assumptions but do not seem to yield further regularity results beyond the Lipschitz regularity. In this work we make the choice  to  use this second  method   (see also \cite{dem11,IJS17,IS} for further applications). For the Lipschitz estimates,  we avoid the Bernstein method which would require a higher regularity of the source term $\bar f$  (see \cite{barles,LU68, Lieb96, manfredini}).

\begin{lemma}\label{gort}
Let $-1<\gamma<\infty$, $1<p<\infty$ and $f\in C(Q_1)\cap L^\infty(Q_1)$.		Let $u$ be a bounded viscosity solution to equation \eqref{maineq}. There exists a constant $C=C(p,n,\gamma, \beta)>0$ such that   for all  $(x,t),(y,t)\in Q_{7/8}$, it holds
		\begin{equation}
		\begin{split}
		\abs{ u(x,t)-u(y,t)}\leq C \left(\norm{  u}_{L^\infty(Q_1)}+\norm{  u}_{L^\infty(Q_1)}^{\frac{1}{1+\gamma}}+\norm{f}_{L^\infty(Q_1)}^{\frac{1}{1+\gamma}}\right) \abs{x-y}.
		\end{split}
		\end{equation}
	\end{lemma}
	 In order to keep the paper easy to read, the proofs are  postponed to Section \ref{sect6}.\\

Next, we use an extension of Kruzhkov's regularity theorem in time (see \cite{Kr69}) to provide a uniform control on the Hölder norm with respect to the time variable. This method was used in \cite{barlesbile,IJS17,gild}  and is based on the interplay between the regularity in time and space.  One could adapt the argument of \cite{IJS17} after tacking into account the source term. Here we give the details for completeness.

\begin{lemma}\label{regutime1}
Assume that $-1<\gamma<\infty$, $1<p<\infty$ and $f\in C(Q_1)\cap L^\infty(Q_1)$.  Let $u$ be a viscosity solution to \eqref{maineq} with $\underset{Q_1}{\osc} \, u\leq A$. Then there exists a constant $C=C(p,n, \gamma, A, \norm{f}_{ L^\infty(Q_1)})>0$ such that
\begin{equation}
\underset{\substack{(x,t), (x,s)\in Q_{11/16},\\
s\neq t}}{\sup}\,\dfrac{|u(x,t)-u(x,s)|}{|t-s|^{\nu}}\leq C,
\end{equation}
where $\nu:=\min\left(\frac12, \frac{1}{2+\gamma}\right)$.
\end{lemma}

	\begin{proof}

We denote by $\C_{Lip}$ the Lipschitz constant obtained in Lemma \ref{gort}  for the solution $u$.
We start by the case $\gamma\geq 0$.
 We aim to show that for all $t_0\in [-(11/16)^2, 0)$ and $\eta>0$, we can find positive constants $M_1, M_2$
 such that the function
$$\bar v(x,t):=u(0,t_0)+M_1(t-t_0)+\frac{M_2}{\eta}|x|^2+\eta$$ is a supersolution to \eqref{maineq}  in $B_{11/16}\times (t_0, 0)$  and satisfies  $u\leq \bar v$ on $\partial_p (B_{11/16}\times(t_0, 0])$.
Using the boundedness and the spatial Lipschitz regularity of $u$, we have  for all $x\in B_{11/16}$
$$u(x,t_0)-u(0, t_0)\leq  \C_{Lip}|x|\leq  \frac{\C_{Lip}^2}{\eta}|x|^2+\eta,$$
and for $(x,t) \in \partial B_{11/16}\times[t_0, 0]$ we have
$$u(x,t)-u(0, t_0)\leq  2\norm{u}_{L^\infty (Q_{11/16})}\leq 
2\norm{u}_{L^\infty(Q_1)} \frac{16}{11}|x|\leq  \left(\frac{32}{11}\right)^2\frac{\norm{u}_{L^\infty(Q_1)}^2}{\eta}|x|^2+\eta.$$
By taking $$M_2=\C_{Lip}^2+\left(\frac{32}{11}\right)^2\norm{u}_{L^\infty(Q_1)}^2,$$ we have that 
$u\leq \bar v$ on $\partial_p(B_{11/16}\times(t_0, 0])$.
Next,  taking 
$$M_1= \eta^{-(\gamma+1)}M_2^{\gamma+1}C(n,p)+\norm{ f}_{L^\infty(Q_1)},$$
 we get  that  the function $\bar v$ satisfies in the viscosity sense
\begin{align*}
\partial_t \bar v=M_1&=\eta^{-(\gamma+1)}M_2^{\gamma+1}C(n,p)+\norm{ f}_{L^\infty(Q_1)}\\
&\geq |D\bar v|^\gamma\Delta_p^N \bar v+f.
\end{align*}
It follows  that $\bar v$ is a supersolution  of \eqref{maineq} in $B_{11/16}\times(t_0, 0]$, and by the maximum principle, we get that 
$u(x,t)\leq \bar v(x,t)$ for $(x,t)\in B_{11/16}\times[t_0, 0]$.
Consequently, for any $\eta>0$, we have
\begin{align*}
	u(0,t)-u(0, t_0)&\leq \dfrac{C(p,n) M_2^{\gamma+1}(t-t_0)}{\eta^{\gamma+1}}+
	\norm{ f}_{L^\infty(Q_1)}(t-t_0)+\eta.
	\end{align*}
Taking $\eta=(\norm{u}_{L^\infty(Q_1)}^2+\C_{Lip}^2)^{\frac{\gamma+1}{\gamma+2}}|t-t_0|^{\frac{1}{\gamma+2}}$ in this inequality, we end up with
	\begin{align*}
	u(0,t)-u(0, t_0)	&\leq C(p,n)(\norm{u}_{L^\infty(Q_1)}^2+\C_{Lip}^2)^{\frac{\gamma+1}{\gamma+2}}|t-t_0|^{\frac{1}{2+\gamma}}+\norm{ f}_{L^\infty(Q_1)}|t-t_0|\\
	&\leq C(p,n,\gamma, A, \norm{f}_{L^\infty(Q_1)})|t-t_0|^{\frac{1}{2+\gamma}}.
	\end{align*}
 The lower  bound follows by comparing with similar barriers and we get the desired result.

\noindent For $\gamma<0$, consider  the function
$$\bar w(x,t):=u(0,t_0)+M_1(t-t_0)+\frac{M_2}{\eta^{\frac{\gamma+2}{\gamma+1}}}|x|^{\frac{\gamma+2}{\gamma+1}}+\eta^{\gamma+2}.$$ 
Using the boundedness and the spatial Lipschitz regularity of $u$, we have  for all $x\in B_{11/16}$
$$u(x,t_0)-u(0, t_0)\leq  \C_{Lip}|x|\leq \C_{Lip}^{\frac{\gamma+2}{\gamma+1}}\left(\frac{|x|}{\eta}\right)^{\frac{\gamma+2}{\gamma+1}}+\eta^{\gamma+2},$$
and for $(x,t) \in \partial B_{11/16}\times[t_0, 0]$, we have
\begin{align*}
u(x,t)-u(0, t_0)&\leq  2\norm{u}_{L^\infty (Q_{11/16})}\\
&\leq 
2\norm{u}_{L^\infty(Q_1)} \frac{16}{11}|x|\\
&\leq  \left(\frac{32}{11} \norm{u}_{L^\infty(Q_1)}\right)^{\frac{\gamma+2}{\gamma+1}}\left(\frac{|x|}{\eta}\right)^{\frac{\gamma+2}{\gamma+1}}+\eta^{\gamma+2}.
\end{align*}
So taking $$M_2=\left(\C_{Lip}+\frac{32}{11} \norm{u}_{L^\infty(Q_1)}\right)^{\frac{\gamma+2}{\gamma+1}},$$ we have that 
$u\leq \bar w$ on $\partial_p(B_{11/16}\times(t_0, 0])$.
Next, taking
	$$M_1= \eta^{-(\gamma+2)}M_2^{\gamma+1}C(n,p)+\norm{ f}_{L^\infty(Q_1)},$$
 we get  that  the function $\bar w$ satisfies in the viscosity sense
\begin{align*}
\partial_t \bar w=M_1&=\eta^{-(\gamma+1)}M_2^{\gamma+1}C(n,p)+\norm{ f}_{L^\infty(Q_1)}\\
&\geq |D\bar w|^\gamma\Delta_p^N \bar w+ f.
\end{align*}
From the comparison principle we get that
$u(x,t)\leq \bar w(x,t)$ for $(x,t)\in B_{11/16}\times[t_0, 0]$.
Hence, for any $\eta>0$, we have
\begin{align*}
	u(0,t)-u(0, t_0)&\leq \dfrac{C(p,n) M_2^{\gamma+1}(t-t_0)}{\eta^{\gamma+2}}+
	\norm{ f}_{L^\infty(Q_1)}(t-t_0)+\eta^{\gamma+2}.
	\end{align*}
Taking $\eta= M_2^{\frac{\gamma+1}{2(\gamma+2)}}|t-t_0|^{\frac{1}{2(\gamma+2)}}=(\norm{u}_{L^\infty(Q_1)}^2+\C_{Lip}^2)^{\frac{\gamma+1}{\gamma+2}}|t-t_0|^{\frac{1}{2(\gamma+2)}}$ in this inequality, we get
	\begin{align*}
	u(0,t)-u(0, t_0)	&\leq C(p,n)(\norm{u}_{L^\infty(Q_1)}+\C_{Lip})^{\frac{\gamma+1}{2}}|t-t_0|^{\frac{1}{2}}+\norm{ f}_{L^\infty(Q_1)}|t-t_0|\\
	&\leq C\left(p,n,\gamma, A, \norm{f}_{L^\infty(Q_1)}\right)|t-t_0|^{\frac{1}{2}}.
	\end{align*}
	 The lower  bound follows by comparing with similar barriers. The proof is complete once we recall the dependence of $\C_{Lip}$.
\end{proof}
	 \noindent{\bf Uniform Lipschitz estimates for deviation from planes}
 For $\gamma\geq 0$, we claim  the following uniform  Lipschitz estimates for solutions to \eqref{devia}.  In the singular case, it is not clear if it is possible to provide a similar result.

\begin{lemma}\label{Liphom1}
Assume that  $0\leq\gamma<\infty$, $1<p<\infty$ and  $\bar f\in L^\infty(Q_1)\cap C(Q_1)$. Let $w$ be a bounded viscosity solution of \eqref{devia}. Then there exists a constant $C=C(p,n, \gamma)$ such that  for all $(x,t),(y,t)\in Q_{3/4}$, it holds
\begin{equation}
		\begin{split}
		\abs{ w(x,t)-w(y,t)}\leq C \left(1+\norm{  w}_{L^\infty(Q_1)}+\norm{\bar f}_{L^\infty(Q_1)}\right) \abs{x-y}.
		\end{split}
		\end{equation}
\end{lemma}
The proof  follows from classical  long and tedious computations and it is   postponed to  Section \ref{sect7}.  There are already many Lipschitz estimates in the literature for related equations. Here the main difficulty is to track the dependence on $q$ and to provide an estimates which does not depend on $|q|$, especially for $|q|$ large. These uniform estimates will play a crucial role in Section \ref{sect5}.\\

\noindent{\bf Fixing the constants.} We  denote by 
$$C_1=C_1(p,n,\gamma):=C
		 \left(1+\norm{  w}_{L^\infty(Q_1)}+\norm{\bar f}_{L^\infty(Q_1)}\right) $$
		 the Lipschitz constant coming from Lemma \ref{Liphom1} where we fix
		  $\norm{ \bar f}_{L^\infty(Q_1)}=2$ and  $\norm{  w}_{L^\infty(Q_1)}=4$. We set 
		 $$C_2=C_2(p,n, \gamma):=2C_1$$
		 and 
		 $$A_1=A_1(p,n,\gamma):=1+2C_2.$$
		 These constants will appear in Section \ref{sect4} and Section \ref{sect5}.
	
\section{First alternative and improvement of flatness}\label{sect4}

In this section we study a first alternative that corresponds to the Degenerate  Alternative. In this case we  show how to improve the flatness of the solution   in a suitable inner cylinder and how one can iterate this improvement.
\subsection{Approximation}
First we state the improvement of flatness property for solutions to  \eqref{maineq}  and  provide its proof. Our main task consists in showing a linear approximation result for solutions to equation \eqref{maineq}. We proceed by contradiction, using the previous lemmas together with the result of  \cite{IJS17} for the homogeneous equation associated to \eqref{maineq}.
	
	\begin{lemma}[The approximation Lemma]\label{approl}
Let $-1<\gamma<\infty$ and $1<p<\infty$. Let $u$ be a viscosity solution to \eqref{maineq} with $\underset{Q_1}{\osc} \, u\leq A_1=A_1(p,n, \gamma)$.	For every $\eta >0$, there exists  $\eps_0(p,n, \gamma, \eta)\in (0, 1)$ such that  if $\norm{f}_{L^\infty(Q_1)}\leq \eps_0$, then there exists a solution $\bar u$ to 
	\begin{equation}\label{homeq1}\partial_t\bar u=|D\bar u|^\gamma\Delta_p^N\bar u\quad\text{in}\quad Q_{11/16}
	\end{equation}
	such that 
	$$\norm{u-\bar u}_{L^\infty(Q_{5/8})}\leq \eta.$$	
	\end{lemma}
	\begin{proof}
	
	We argue by contradiction. We assume that there exist a constant $\eta>0$ and 
sequences $\eps_k\to 0$,  $f_k$  and $u_k$ such that $\norm{f_k}_{L^\infty(Q_1)} \leq \eps_k$, and $u_k$ are  solutions in $Q_1$  of 
$$\partial_t u_k=|Du_k|^\gamma\Delta_p^N u_k+f_k,$$	
with $\underset{Q_1}{\osc} \, u_k\leq A_1$ and such that 
\begin{equation}\label{claude}
\norm{u_k-\bar u}_{L^\infty(Q_{5/8})}>\eta
\end{equation}
whenever $\bar u $ is  a viscosity solution to  \eqref{homeq1}.

\noindent Combining  the compactness estimates coming from Lemma  \ref{gort} and Lemma \ref{regutime1} with  the Arzelà-Ascoli's Lemma, we can extract a subsequence of $(u_k)_k$ converging  uniformly in $Q_\rho$ for any $\rho\in (0, 11/16)$ to a continuous function $u_\infty$.  Passing to the limit in the equation, we have that $u_\infty$ solves in $Q_{11/16}$
$$\partial_tu_\infty -|Du_\infty|^\gamma\left[\Delta u_\infty+(p-2)\dfrac{\langle D^2u_\infty  Du_\infty, D u_\infty\rangle}{|D u_\infty|^2}   \right]=0.$$
We end up with a contradiction for \eqref{claude}, since for $k$ large enough we have 
$$\norm{u_k-u_\infty}_{L^\infty(Q_{5/8})}\leq \eta.$$
	\end{proof}

Combining the approximation lemma and the  regularity result of Theorem \ref{imbert}, we  can now state the following improvement of flatness lemma. 
	Define 
$$Q_{\rho}^{(1-\delta)}:=B_{\rho}\times (-\rho^{2}(1-\delta)^{-\gamma}, 0].$$

	\begin{lemma}\label{starting}
	Let $-1<\gamma<\infty$ and $1<p<\infty$. Let $u$ be a viscosity solution to \eqref{maineq} such that $\underset{Q_1}{\osc}\, u\leq A_1=A_1(p,n, \gamma)$. There exist $\eps_0=\eps_0(p,n, \gamma)>0$, $\rho=\rho(p,n, \gamma)>0$  and $\delta=\delta(p,n, \gamma)\in (0,1)$ with $\rho<(1-\delta)^{\gamma+1}$ such that if $\norm{f}_{L^\infty(Q_1)} \leq \eps_0$, then there exists a vector $h$ with $|h|\leq B=B(n, p, \gamma)$ such that
	$$\underset{(x,t)\in Q_{\rho}^{(1-\delta)}}{\osc}\, (u(x,t)-h\cdot x)\leq \rho(1-\delta).$$

	\end{lemma}
	\begin{proof}
	
	Let $\bar u$ be the viscosity solution to 
	$$\partial_t \bar u-|D \bar u|^\gamma\Delta_p^N \bar u=0\qquad\text{in}\qquad Q_{11/16}$$
coming from Lemma \ref{approl}.
	From the  regularity result of  Theorem \ref{imbert}, there exists $C=C(p,n, \gamma)>0$ such that for all $\mu\in (0, 5/8)$ there exists
 $h$ with $|h|\leq B=B(p,n, \gamma)$ such  that
	$$\underset{(x,t)\in Q_{\mu}}{\osc}\, (\bar u(x,t)-h\cdot x)\leq C(p,n, \gamma)(1+\norm{\bar  u}_{L^\infty(Q_{5/8})})\mu^{1+ \beta_1}.$$
	We pick   a $\mu_0=\mu_0(p,n, \gamma)\in (0, 5/8)$ such that 
$$\underset{(x,t)\in Q_{\mu_0}}{\osc}\, (\bar u(x,t)-h\cdot x)	\leq \frac{1}{4}\mu_0(1-\delta)^{\gamma+2}$$
for some $\delta=\delta(p,n\gamma)\in (0,1)$.
Thus there exist	two constants $\rho$ and $\delta$ depending on $p, n, \gamma$ such that
	$$\underset{(x,t)\in Q_{\rho}^{1-\delta}}{\osc}\, (\bar u(x,t)-h\cdot x)\leq \frac{1}{4} \rho (1-\delta)$$
	with $\rho=\mu_0(1-\delta)^{\gamma+1}<(1-\delta)^{\gamma+1}$.\\
	\noindent It follows from Lemma \ref{approl}, that for $\eta:=\frac{1}{4} \rho(1-\delta)$ there exists $\eps_0$ such that if  $\norm{f}_{L^\infty(Q_1)}\leq \eps_0$,  we have
	\begin{align*}
	\underset{(x,t)\in Q_{\rho}^{1-\delta}}{\osc}\, (u(x,t)-h\cdot x)
	&\leq \underset{(x, t)\in Q_{\mu_0}}{\osc}\, (u(x,t)-\bar u(x,t))+\underset{(x,t)\in Q_{\rho}^{1-\delta}}{\osc}\, (\bar u(x,t)-h\cdot x)\\
	&\leq \eta+\frac{1}{2} \rho(1-\delta)
	\leq  \rho(1-\delta).
	\end{align*}
	The choice of $\eta$ determines the smallness of $f$.
	\end{proof}

		\subsection{ Iteration and condition for alternatives}
		Here we study the situation when the
Degenerate  Alternative holds a certain number of times when considering a suitable
chain of shrinking intrinsic cylinders. 
\begin{lemma}[Improvement of flatness]\label{iterimp}
Let $-1<\gamma<\infty$ and $1<p<\infty$. Let $u$ be a viscosity solution to \eqref{maineq} such that $\osc_ {Q_1} u\leq 1$.   Assume  that $\norm{f}_{L^\infty(Q_1)} \leq \eps_0$, where   $\eps_0=\eps_0(p,n, \gamma)>0$, is the constant appearing in Lemma \ref{starting}. Then there exist  $\rho=\rho(p,n, \gamma)>0$  and $\delta=\delta(p,n, \gamma)\in (0,1)$ with $\rho<(1-\delta)^{\gamma+1}$ such that, if  for every nonnegative integer $k$ it holds
 \begin{equation}\label{conditionimpro}
\quad\left\{\begin{array}{ll}
\exists\, \, l_i \quad\text{with}\quad |l_i|\leq C_2(1-\delta)^i\quad \text{such that}\,\,\underset{(x,t)\in Q_{\rho^i}^{(1-\delta)^i}}{\osc}\, (u(x,t)-l_i\cdot x)\leq \rho^i(1-\delta)^i\\
 \text{for}\quad i=0,\ldots , k,
\end{array}
\right.
\end{equation}
then there exists a vector $l_{k+1}$ such that 
$$\underset{(x,t)\in Q_{\rho^{k+1}}^{(1-\delta)^{k+1}}}{\osc}\, (u(x,t)-l_{k+1}\cdot x)\leq \rho^{k+1}(1-\delta)^{k+1}$$
and $$|l_{k+1}-l_k|\leq C_3(1-\delta)^k,$$
with $C_3=C_3(p,n, \gamma)>0$.

\end{lemma}
\begin{proof}
Let $\rho, \delta, B$ and $\eps_0$ be the constants coming from Lemma \ref{starting} and let 
$$C_3:=B+ C_2.$$
For $j=0$ we take  $l_0=0$, and  the result follows from Lemma \ref{starting},  since  $\osc_{Q_1} u\leq 1<A_1$. Suppose that the result of the Lemma holds true for $j= 0, \ldots , k $. We are going to  prove it for $ j= k+1$. Let
$$w_k(x,t):=\frac{u(\rho^k x, \rho^{2k}(1-\delta)^{-k\gamma} t)-l_k\cdot \rho^k x}{\rho^k(1-\delta)^k},$$
and denote $\bar f(x,t):=\rho^k(1-\delta)^{-k(\gamma+1)} f(\rho^k x, \rho^{2k}(1-\delta)^{-k\gamma} t)$.
By assumption,  we have that $\osc_{Q_1} w_k\leq 1$ and $|l_k|\leq C_2 (1-\delta)^k$. 
Let $q=\frac{l_k}{(1-\delta)^k}$. Now the function $\bar v(x,t):= w_k(x,t)+q\cdot x$   satisfies 
$$\osc_{Q_1} \bar v\leq 1+2|q|\leq 1+2C_2\leq A_1.$$
Moreover,  $\bar v$ solves in $Q_1$
$$\partial_t\bar v=|D\bar v|^\gamma\Delta_p^N\bar v+\bar f,$$
with 
$$\norm{\bar f}_{L^\infty(Q_1)}\leq \eps_0<1.$$
Here we used   that $\rho(1-\delta)^{-(\gamma+1)}<1$. Therefore, by Lemma \ref{starting} we have that there exists $h$ with $|h|\leq B=B(p,n, \gamma)$ such that 
$$\underset{(x,t)\in Q_{\rho}^{(1-\delta)}}{\osc}\, (\bar v(x,t)-h\cdot x)\leq \rho(1-\delta).$$
Going back to $u$, we have
	$$\underset{(x,t)\in Q_{\rho}^{(1-\delta)}}{\osc}\, (u(\rho^k x, \rho^{2k}(1-\delta)^{-k\gamma} t)-\rho^k(1-\delta)^k h \cdot x)\leq \rho^{k+1}(1-\delta)^{k+1}.$$
	Scaling back, we get that 
	$$\underset{(x,t)\in Q_{\rho^{k+1}}^{(1-\delta)^{k+1}}}{\osc}\, (u(x,t)-l_{k+1}\cdot x)\leq \rho^{k+1}(1-\delta)^{k+1},$$
	where 
	$$l_{k+1}:= (1-\delta)^k h$$
	satisfies 
	$$|l_{k+1}-l_k|\leq( B+C_2)(1-\delta)^k= C_3(1-\delta)^k.$$
\end{proof}

	\section{Handling the two alternatives  and proof of the main theorem}\label{sect5}
In this section, we assume that $\gamma\geq 0$. We prove the Hölder continuity of  $Du$  at the
origin and   the improved Hölder regularity of $u$ with respect to the time variable. Then the result follows by standard translation and scaling  arguments.  The H\"older regularity with respect to the space variable  is a direct consequence of  the following lemma  after scaling back from $u_\theta$  to $u$. 
	
\begin{theorem}

Let $0\leq \gamma<\infty$, $1<p<\infty$ and let $u$ be a viscosity solution to \eqref{maineq} with $\underset{Q_1}{\osc}\, u\leq 1$. Let $\eps_0$ be the constant coming from Lemma \ref{starting} and assume that $\norm{f}_{L^\infty(Q_1)}\leq \eps_0$. Then there exist $\alpha=\alpha(p,n, \gamma)\in \left(0, \frac{1}{\gamma+1}\right)$ and $C=C(p,n, \gamma)>0$ such that
$$|Du(x,t)-Du(y,s)|\leq C(|x-y|^\alpha+|t-s|^{\alpha/2})$$
and 
$$|u(x,t)-u(x,s)|\leq C|t-s|^{(1+\alpha)/2}.$$

\end{theorem}
\begin{proof}
Let $\rho$ and $\delta$  be the constants coming from Lemma \ref{starting}.
Let $k$ be  the minimum nonnegative integer such that
the condition \ref{conditionimpro} does not hold.
We can conclude from Lemma \ref{iterimp} that for any vector  $\xi$ with $|\xi|\leq C_2(1-\delta)^k$, it holds 
\begin{equation}
|u(t,x)-\xi\cdot x|\leq C(|x|^{1+\tau}+|t|^{\frac{1+\tau}{2-\tau\gamma}})\quad \text{for}\quad (x,t)\in Q_1\setminus Q_{\rho^{k+1}}^{(1-\delta)^{k+1}},
\end{equation}
	where $\tau:=\frac{\log(1-\delta)}{\log(\rho)}$
and $C=\frac{1+C_2+ C_3(1-\delta)^{-1}}{\rho(1-\delta)}$.
Now we treat differently the following two cases.\\

\noindent {\bf First case:} If $k=\infty$, then the regularity result holds with $$\alpha=\min(1,\tau)=\min\left(1,\dfrac{\log(1-\delta)}{\log \rho}\right)\in \left(0, \min\left(1,\frac{1}{1+\gamma}\right)\right).$$
Indeed, for all $k\in\N$, there exists $l_k\in \R^n$  with $|l_k|\leq C_2 (1-\delta)^k$ such that
		\begin{equation*}
		\underset{(y, t)\in Q^{(1-\delta)^k}_{\rho^{k}}}{\osc} \, (u(y,t)-l_{k}\cdot y)\leq  \rho^k(1-\delta)^k,
		\end{equation*}
	and the conclusion follows  using the characterization of functions with Hölder continuous gradient \cite{Lieb96}.\\
	
\noindent {\bf Second  case:} If $k<\infty$, then it follows from Lemma 4.2 that for
all $i =  0, \ldots, k$, we have  the existence of vectors $l_i$ such that 

\begin{equation}\label{iteratio2}
		\underset{(y, t)\in Q^{(1-\delta)^i}_{\rho^{i}}}{\osc} \, (u(y,t)-l_{i}\cdot y)\leq  \rho^i(1-\delta)^i,
		\end{equation}
		with 
		\begin{align*}
		&|l_{i}|\leq C_2 (1-\delta)^{i}\qquad\text{for}\,\, i=0, \ldots,k-1,\\
		&|l_k-l_{k-1}|\leq C_3(1-\delta)^{k-1},\\
		&|l_k|\geq C_2(1-\delta)^k.
		\end{align*}
		It follows that
\begin{equation}\label{geremy}
		\underset{(y, t)\in Q^{(1-\delta)^k}_{\rho^k}}{\osc} \, (u(y,t)-l_{k}\cdot y)\leq  \rho^k (1-\delta)^k
		\end{equation}
		and 
		\begin{equation}\label{eric}
		C_2(1-\delta)^k\leq |l_k|\leq (C_3+C_2)(1-\delta)^{k-1}.
		\end{equation}
Consider   for $(x,t)\in Q_1$ the function
		$$v(x,t):=\frac{u(\rho^k x, \rho^{2k}(1-\delta)^{-k\gamma} t)}{\rho^k(1-\delta)^k}.$$
 Then $v$ satisfies  
 $$\partial_t v=|Dv|^\gamma\Delta_p^N v+ \bar f,$$
 where 
 $$\bar f(x, t):= \rho^k(1-\delta)^{-k(1+\gamma)}f(\rho^k x, \rho^{2k}(1-\delta)^{-k\gamma} t).$$
 Notice that  we can write $$v(x,t)=v(x,t)-q\cdot x+q\cdot x:=w(x,t)+q\cdot x,$$
 where 
 $q:=\frac{l_k}{(1-\delta)^k}$ and 
 $$w(x,t):=\frac{u(\rho^k x, \rho^{2k}(1-\delta)^{-k\gamma} t)-l_k\cdot \rho^k x}{\rho^k(1-\delta)^k}.$$
Observe that  by assumption, $w$  satisfies  $\underset{Q_1}{\osc}\, w\leq 1$ and  solves in $Q_1$

$$\partial_t w-|Dw+q|^\gamma\left[\Delta w+(p-2) \left\langle D^2w\frac{Dw+q}{\abs{Dw+q}}, \frac{Dw+q}{\abs{Dw+q}}\right\rangle\right]=\bar f$$
where, due to   $\rho<(1-\delta)^{\gamma+1}$, we have    $\norm{\bar f}_{L^\infty(Q_1)}\leq \eps_0\leq 1$.
From Lemma \ref{Liphom1}, we have a uniform Lipschitz bound for $w$:
 $$|Dw(x,t)|\leq C_1 \quad\text{for}\quad (x, t)\in Q_{3/4}.$$
It follows that  for  $(x,t) \in Q_{3/4}$, we have
 \begin{equation}\label{inegal1}
 |Dv(x,t)|=|Dw(x,t)+q|\geq |q|- |Dw(x,t)|\geq C_1,
 \end{equation}
 where we used that  from \eqref{eric}  we have $|q|= \left|\frac{l_k}{(1-\delta)^k}\right|\geq C_2=2C_1$.
 Moreover, using the upper bound on $|l_k|$ coming from \eqref{eric}, we have
 \begin{align}\label{boun}\norm{v}_{L^\infty(Q_1)}&=\norm{w+\frac{l_k}{(1-\delta)^k}\cdot x}_{L^\infty(Q_1)}\nonumber\\
 &\leq \norm{w}_{L^\infty(Q_1)}+\left|\frac{l_{k}}{(1-\delta)^k}\right|\nonumber\\
 &\leq 2+ (C_3+C_2)(1-\delta)^{-1}:= C_4(p,n, \gamma).
 \end{align}
  From the local Lipschitz estimate coming from Lemma \ref{gort} and the estimate \eqref{boun}, it follows that 
  \begin{equation}\label{inegal2}\norm{Dv}_{L^\infty(Q_{3/4})}\leq C(p,n, \gamma)\left(1+\norm{v}_{L^\infty(Q_1)}+\norm{\bar f}_{L^\infty(Q_1)}\right)\leq C_5(p,n, \gamma).
  \end{equation}
  
  Combining the estimates  \eqref{inegal1} and \eqref{inegal2}, we notice  that $v$ solves a uniformly parabolic equation with ellipticity constants depending only on $p,n, \gamma$ and that this equation is smooth in the gradient variables.
  It follows that $v$ has a Hölder continuous gradient with a Hölder norm  and a Hölder exponent depending only on $p,n, \gamma$.
  Indeed, from Theorem \ref{ladyhope}, we get that $v\in C^{1+\bar \alpha, (1+\bar\alpha)/2}_{loc}(Q_{3/4})$ for some $\bar \alpha = \bar \alpha(p,n,\gamma)>0$ and 
$$\norm{Dv}_{C^{\bar\alpha}(\omega)}\leq C\left(p,n, \gamma, \dist(\omega, \partial_pQ_{3/4})\right)$$
for any $\omega\subset \subset Q_{3/4}$.
Let  $ 0<\alpha\leq\min\left(\bar\alpha, \frac{\log(1-\delta)}{\log(\rho)}\right)$.
Hence,  there exists a vector $l\in \R^n$ such that  in $Q_\rho^{1-\delta}$ we have
$$|Dv(x, t)-l|\leq C(|x|^{ \alpha}+|t|^{\frac\alpha 2})$$
and 
$$|v(x,t)-v(x,0)|\leq C|t|^{\frac{1+\alpha}{2}}.$$
Coming back to  $u$, it follows that in $Q_{\rho^{k+1}}^{(1-\delta)^{k+1}}$, it holds
\begin{align*}|Du(y,s)-(1-\delta)^k l|&\leq C(\rho^{-k \alpha}(1-\delta)^k|y|^{ \alpha}+(1-\delta)^k(\rho^{-2}(1-\delta)^{\gamma})^{\frac{k \alpha}{2}}|s|^{\frac \alpha2})\\
&\leq C(|y|^{\alpha}+|s|^{\frac{\alpha}{2}})
\end{align*}
and 
\begin{align}\label{bernard}|u(y,s)-u(y,0)|&\leq C\rho^k(1-\delta)^k |s|^{\frac{1+\alpha}{2}} (\rho^{-2}(1-\delta)^\gamma)^{\frac{k(1+\alpha)}{2}}\nonumber\\
&\leq C|s|^{\frac{1+\alpha}{2}}.
\end{align}
Here we used that $\rho^{-\alpha}(1-\delta)\leq 1$ due to  $0<\alpha\leq \frac{\log{1-\delta)}}{\log{\rho}}$.

\noindent Consequently, combining these estimates with \eqref{iteratio2}, we have showed that  for $$0<\alpha=\min \left(\bar \alpha, \frac{\log(1-\delta)}{\log(\rho)}\right),$$ there exists a constant $C=C(p,n, \gamma)$ such that, for any $r\leq \frac{1}{2}$, there exists a vector $V=V(r)$ such that
\begin{equation}\label{bousa}|u(x,t)-u(0,0)-V\cdot x|\leq Cr^{1+\alpha}
\end{equation}
whenever $|x|+\sqrt{|t|}\leq r$.
The regularity of $Du$ follows then from \cite[Lemma 12.12]{Lieb96}.

The Hölder regularity  of $u$ in time  follows from \eqref{iteratio2},\eqref{geremy},\eqref{eric} and  \eqref{bernard}. Indeed,  for $i=0,\ldots k$, we have 

\begin{align}\label{yohan}
\underset{(y, t)\in Q^{(1-\delta)^i}_{\rho^{i}}}{\osc} \, (u(y,t)-u(0,0))&\leq  \underset{(y, t)\in Q^{(1-\delta)^i}_{\rho^{i}}}{\osc} \, (u(y,t)-l_{i}\cdot y)+\underset{(y, t)\in Q^{(1-\delta)^i}_{\rho^{i}}}{\osc} \, l_i\cdot y\nonumber\\
&\leq  \rho^i(1-\delta)^i+2|l_i|\rho^i\\
&\leq \big[1+2(C_2+C_3)(1-\delta)^{-1}\big]\rho^i(1-\delta)^i\nonumber\\
&\leq C(p,n, \gamma)\rho^i(1-\delta)^i.\nonumber
\end{align}
Combining the estimate \eqref{yohan} with \eqref{bernard}, we obtain that for $-1/4\leq t\leq 0$, it holds
\begin{equation*}
|u(0,t)-u(0,0)|\leq C(p,n, \gamma)|t|^{\frac{1+\alpha}{2}}.
\end{equation*}

  \end{proof}

\begin{remark}
 One could  use another alternative to improve the Hölder regularity in time  knowing the  Hölder regularity of the spatial derivative, based on  a parabolic maximum principle and barriers.  Indeed, one  can  modify the arguments  used in the proof of Lemma \ref{regutime1} or use the  argument in \cite[proof of Theorem 4.5]{IJS17}. For related works, we refer to \cite{Ar15, IJS17}. It would be useful to know when it is possible to provide the local  regularity in time without using the regularity in space (and without assuming much regularity on the initial data) and how  it would then imply the higher regularity in space. This was done for the Cauchy problem for a class of parabolic equations in  \cite{bourg}.
\end{remark}

\section{Proofs of the local  Hölder and Lipschitz  regularity for solutions to \eqref{maineq}}\label{sect6}
In this section we provide Hölder and Lipschitz estimates for viscosity solutions to \eqref{maineq}.
These estimates are valid for the degenerate and the  singular case. We assume that $-1<\gamma<\infty$ and $1<p<\infty$.
The proof follows roughly the same lines as the one in \cite{AP17,ARP,dem11,IJS17}. We aim at proving that the maximum
$$\underset{(x,t)\in Q_r}{\max}\, (u(x,t)-u(y,t)-\vp(|x-y|))$$
is non-negative, choosing in a first step $\vp(s) = Ls^\beta$ with $\beta \in (0, 1)$, to obtain a Hölder bound, and  in a second step, $\vp(s) = L(s -\kappa_0s^{\nu})$, to improve the Hölder bound into a Lipschitz
one. To do this, we argue by contradiction and  we use in a crucial way the strict concave behavior of $\vp$ near 0 to take
profit of the strict ellipticity of the equation as usual in Ishii-Lions' method.
We first show  $C_x^{0, \beta}$ estimates for all $\beta\in(0,1)$, then we check that  this implies the  Lipschitz continuity estimates. 
\subsection{Local Hölder estimates}
\begin{lemma}\label{holdnormsol}
Let $-1<\gamma<\infty$ and $1<p<\infty$.
		Let $u $ be a bounded viscosity solution to equation \eqref{maineq}. For any $\beta\in (0,1)$, there exists a constant $C=C(p,n,\gamma, \beta)>0$ such that   for all  $x,y\in B_{15/16}$ and $t \in (-(15/16)^2,0]$, it holds
		\begin{equation}
		\begin{split}
		\abs{u(x,t)-u(y,t)}\leq C \left(\norm{ u}_{L^\infty(Q_1)}+\norm{ u}_{L^\infty(Q_1)}^{\frac{1}{1+\gamma}}+\norm{ f}_{L^\infty(Q_1)}^{\frac{1}{1+\gamma}}\right) \abs{x-y}^\beta.
		\end{split}
		\end{equation}
		
	\end{lemma}
	\begin{proof}

We fix $x_0, y_0\in B_{15/16}$, $t_0\in (-(15/16)^2,0)$. For suitable  constants $L_1, L_2>0$, we define the auxiliary function 
		\begin{align*}
		\Phi(x, y,t):&=u(x,t)-u(y,t)-L_2\vp(\abs{x-y})-\frac{L_1}{2}\abs{x-x_0}^2-
		\frac{L_1}{2}\abs{y-y_0}^2-\frac{L_1}{ 2} (t-t_0)^2,
		\end{align*}
		where $\vp(s)= s^\beta$. We want to show that $\Phi(x, y,t)\leq 0$ for $(x,y)\in \overline{B_{15/16}}\times\overline{ B_{15/16}}$ and $t\in [-(15/16)^2,0]$.
	We point out that the role of the term $-\frac{L_1}{2}\abs{x-x_0}^2-\frac{L_1}{2}\abs{y-y_0}^2-\frac{L_1}{2}  (t-t_0)^2$ 
is to localize, while the term $-L_2\vp(\abs{x-y})$
 is concerned with the Hölder continuity. The main idea of the proof relies on  the concavity of $\varphi$ to create a large negative term in the viscosity inequalities.
  We argue by contradiction.  We assume that
		$\Phi$ has a positive
		maximum at some point $(\bar x, \bar y,\bar t)\in \bar B_{15/16}\times \bar B_{15/16}\times [-(15/16)^2,0]$ and we are going to get a contradiction for $L_2$, $L_1$
large enough.
The positivity of the maximum of $\Phi$  implies that $\bar x\neq \bar y$.
		Using the boundedness of   $u$ we can choose
 $$L_1\geq \dfrac{140\osc_{Q_1}{u}}{\min(d\left( (x_0,t_0) ,\partial Q_{15/16}\right), d\left( (y_0,t_0) ,\partial Q_{15/16}\right))^2},$$ so that 
		\begin{equation*} 
\begin{split}
		\abs{\bar y-y_0}+\abs{\bar x-x_0}+|\bar t-t_0|&\leq 2\sqrt{ \frac{2\osc_{Q_1}{u}}{L_1}}\leq \dfrac{d\left( (x_0,t_0),  \partial Q_{15/16}\right)}{2},
		\end{split}
		\end{equation*}
		and hence $\bar x$ and $\bar y$ are in $B_{15/16}$ and $\bar t\in (-(15/16)^2,0)$.

The remaining of the proof is divided into  3 steps.  We write down the viscosity inequalities and  get suitable matrices inequalities from  the  Jensen-Ishii's lemma (Step 1).
Then we estimate  the difference of the obtained terms,  make use  of the concavity of $\varphi$ (Step 2) and finally  we obtain a contradiction and conclude (Step 3).\\

	\noindent {\bf Step 1.}
		By  Jensen-Ishii's lemma (see \cite[Theorem 8.3]{crandall1992user}), there exist
		\[
		\begin{split}
		&(\sigma,\bar\zeta_x,X)\in \overline{\mathcal{P}}^{2,+}\left(u(\bar x,\bar t)-\frac{L_1}{2}\abs{\bar x-x_0}^2-\frac{L_1}{2}(\bar t-t_0)^2\right),\\
		&(\sigma,  \bar \zeta_y,Y)\in \overline{ \mathcal{P}}^{2,-}\left( u(\bar y, \bar t)+\frac{L_1}{2}\abs{\bar y-y_0}^2\right),
		\end{split}
		\]
	which we can rewrite as 
		\[
		\begin{split}
		&(\sigma+L_1(\bar t-t_0), a_1,X+L_1I)\in \overline{\mathcal{P}}^{2,+} u(\bar x, \bar t),\\ &(\sigma, a_2,Y-L_1I)\in \overline{\mathcal{P}}^{2,-} u(\bar y, \bar t),
		\end{split}
		\]
		 where $\bar\zeta_x=\bar \zeta_y$ and 
		\[
		\begin{split}
		a_1&=L_2\varphi'(|\bar x-\bar y|) \frac{\bar x-\bar y}{\abs{\bar x-\bar y}}+L_1(\bar x-x_0)= \bar \zeta_x+L_1(\bar x-x_0),\\
		a_2&=L_2\varphi'(|\bar x-\bar y|) \frac{\bar x-\bar y}{\abs{\bar x-\bar y}}-L_1(\bar y-y_0)= \bar \zeta_y-L_1(\bar y-y_0).
		\end{split}
		\]
		
		Assuming that   $L_2>\dfrac{L_1 2^{4-\beta}}{\beta}$, we have
		\begin{equation}\label{portorico3}
		\begin{cases}
		2L_2\beta \abs{\bar x-\bar y}^{\beta-1}&\geq\abs{a_1}\geq L_2\varphi'(|\bar x-\bar y|) - L_1\abs{\bar x-x_0}\ge \frac{L_2}{2} \beta \abs{\bar x-\bar y}^{\beta-1}\\
		2L_2\beta\abs{\bar x-\bar y}^{\beta-1}&\geq\abs{a_2}\geq L_2\varphi'(|\bar x-\bar y|) - L_1\abs{\bar y-y_0}\ge \frac{L_2}{2} \beta \abs{\bar x-\bar y}^{\beta-1}.
		\end{cases}
		\end{equation}

		\noindent Thanks to Jensen-Ishii's lemma \cite[Theorem 12.2]{crandnote},  we can take $X, Y\in \mathcal{S}^n$ such that for any  $\tau>0$ such that $\tau Z<I$, it holds\\
		\begin{equation}\label{ma1}
		-\frac{2}{\tau} \begin{pmatrix}
		I&0\\
		0&I 
		\end{pmatrix}\leq
		\begin{pmatrix}
		X&0\\
		0&-Y 
		\end{pmatrix}\leq \begin{pmatrix} Z^\tau& -Z^\tau\\
		-Z^\tau& Z^\tau\end{pmatrix},
		\end{equation}
	
		where 
		\begin{align*}	
		Z&=L_2\varphi''(|\bar x-\bar y|) \frac{\bar x-\bar y}{\abs{\bar x-\bar y}}\otimes \frac{\bar x-\bar y}{\abs{\bar x-\bar y}} +\frac{L_2\varphi'(|\bar x-\bar y|)}{\abs{\bar x-\bar y}}\Bigg( I- \frac{\bar x-\bar y}{\abs{\bar x-\bar y}}\otimes \frac{\bar x-\bar y}{\abs{\bar x-\bar y}}\Bigg)\\
		&=L_2\beta\abs{\bar x-\bar y}^{\beta-2}\left(I+(\beta-2)\frac{\bar x-\bar y}{\abs{\bar x-\bar y}}\otimes \frac{\bar x-\bar y}{\abs{\bar x-\bar y}}\right)
		\end{align*}
		and $$Z^\tau= (I-\tau Z)^{-1}Z.$$
		
		\noindent We fix   $\tau=\dfrac{1}{2L_2\beta\abs{\bar x-\bar y}^{\beta-2}}$ so  that we have
		\begin{align*}	
		Z^\tau=(I-\tau Z)^{-1} Z=2L_2\beta\abs{\bar x-\bar y}^{\beta-2}\left(I-2\frac{2-\beta}{3-\beta} \frac{\bar x-\bar y}{\abs{\bar x-\bar y}}\otimes \frac{\bar x-\bar y}{\abs{\bar x-\bar y}}\right).
\end{align*}	
 With this choice for $\tau$, we have  for $\xi=\frac{\bar x-\bar y}{\abs{\bar x-\bar y}}$, 
	
				\begin{equation}\label{oufi11}
		\langle Z^\tau \xi,\xi\rangle= 2L_2\beta\abs{\bar x-\bar y}^{\beta-2}\left(\frac{\beta-1}{3-\beta}\right)<0.
		\end{equation}
Applying the inequality \eqref{ma1} to any  vector $(\xi,\xi)$ with $\abs{\xi}=1$, we get that $X- Y\leq 0$ and 
		\begin{equation}\label{gilout}
		\norm{X},\norm{Y}\leq 4L_2\beta\abs{\bar x-\bar y}^{\beta-2}.
		\end{equation}

\noindent{\bf Step 2.}		
	For $\eta \neq 0$, denoting  $\hat \eta=\dfrac{\eta}{|\eta |}$ and 
		\[\A(\eta):= I+(p-2)\hat\eta\otimes \hat\eta,\] 
		the viscosity inequalities read  as
		\begin{align*}
	L_1(\bar t-t_0)+\sigma 	- f(\bar x, \bar t)&\leq |a_1|^{\gamma} \tr (\A(a_1)(X+L_1I))\\
	-\sigma 	+f(\bar y,\bar t)&\leq  -|a_2|^{\gamma}\tr (\A(a_2)(Y-L_1I)).
		\end{align*}
	Adding the two inequalities and  using that $|\bar t-t_0|\leq 2$, we get  
	
		\begin{align}\label{gregor1}
		0\leq  &2(L_1+|| f||_{L^\infty(Q_1)})+\underbrace{|a_1|^\gamma\tr (\A(a_1)(X-Y))}_{(i_1)}		+\underbrace{|a_1|^\gamma \tr ((\A(a_1)-\A(a_2))Y)}_{(i_2)}\nonumber\\
		&+\underbrace{(|a_1|^\gamma-|a_2|^\gamma)\tr(\A(a_2)Y)}_{(i_3)}+\underbrace{L_1\big[|a_1|^\gamma\tr (\A(a_1))+|a_2|^\gamma\tr (\A(a_2))}_{(i_4)} \big].
		\end{align}
		
In order to estimate ($i_1$), we  use the fact  that all the eigenvalues of $X-Y$ are non positive   and that  at least one of the eigenvalues of $X-Y$  that we denote by  $\kappa_{i_0}(X-Y)$ is   negative and smaller than $8L_2\beta\abs{\bar x-\bar y}^{\beta-2}\left(\frac{\beta-1}{3-\beta}\right)$. Indeed, applying the  matrix inequality \eqref{ma1} to the vector $(\xi,-\xi)$ where $\xi:=\frac{\bar x-\bar y}{|\bar x-\bar y|}$  and using \eqref{oufi11}, 
		we obtain
		
		\begin{align}\label{calle2}
		\langle (X-Y) \xi, \xi\rangle&\leq 4\langle Z^\tau \xi,\xi\rangle\leq 8L_2\beta\abs{\bar x-\bar y}^{\beta-2}\left(\frac{\beta-1}{3-\beta}\right)<0.
		\end{align}
		Noticing that the  eigenvalues of $\A(a_1)$ belong to $[\min(1, p-1), \max(1, p-1)]$ and using \eqref{portorico3} and \eqref{calle2}, we end up with the estimate
		
		\begin{align*}  
		|a_1|^\gamma\tr(\A(a_1) (X-Y))&\leq |a_1|^\gamma\sum_i \kappa_i(\A(a_1))\kappa_i(X-Y)\\
		&\leq |a_1|^\gamma\min(1, p-1)\kappa_{i_0}(X-Y)\\
		&\leq |a_1|^\gamma\min(1, p-1)8L_2\beta\abs{\bar x-\bar y}^{\beta-2}\left(\frac{\beta-1}{3-\beta}\right)\\
		&\leq C\left( L_2 \beta \abs{\bar x-\bar y}^{\beta-1}\right)^\gamma 8L_2\beta\abs{\bar x-\bar y}^{\beta-2}\left(\frac{\beta-1}{3-\beta}\right).
		\end{align*}
In order to estimate  ($i_2$), we decompose
		\begin{align*}\A(a_1)-\A(a_2)&=(\hat a_1\otimes \hat a_1-\hat a_2\otimes \hat a_2)(p-2)=[(\hat a_1-\hat a_2)\otimes\hat a_1
	-\hat a_2\otimes(\hat a_2-\hat a_1)](p-2)
		\end{align*}
		and get 
		\begin{align*}
		\tr( (\A(a_1)-\A(a_2)) Y)\leq n\norm{Y}
		\norm{\A(a_1)-\A(a_2)}  \leq 2n\abs{p-2}\norm{Y}|\hat a_1-\hat a_2|.
		\end{align*}
		Using that $|a_1-a_2|\leq 4 L_1$ and  the estimate \eqref{portorico3}, we get 
		
		\begin{equation*}
		\abs{\hat a_1-\hat a_2}=
		\abs{\frac{a_1}{\abs {a_1}}-\frac{a_2}{\abs {a_2}}}
\le \max\left( \frac{\abs{a_2- a_1}}{\abs{a_2}},\frac{ \abs{a_2- a_1}}{\abs{a_1}}\right)
		\le \frac {16  L_1}{ \beta L_2 \abs{\bar x-\bar y}^{\beta-1}}.
		\end{equation*}
	Recalling  that  $
		\norm{Y}=\underset{\hat \xi}{\max}\, |\langle Y\hat \xi, \hat \xi\rangle|
		\le  4L_2\beta\abs{\bar x-\bar y}^{\beta-2}$,  it follows that 
	\begin{align*}
	|a_1|^\gamma| \tr( (\A(a_1)-\A(a_2)) Y)|&\leq|a_1|^\gamma 128n\abs{p-2}L_1 \abs{\bar x-\bar y}^{-1}\\
	&\leq C\left( L_2 \beta \abs{\bar x-\bar y}^{\beta-1}\right)^\gamma\abs{p-2}L_1 \abs{\bar x-\bar y}^{-1}.
	 \end{align*}
	 Now we estimate the term ($i_3$). Notice that $|a_2|/|a_1|\leq 16$ and $|a_1-a_2|\leq 4L_1$.  Using  the mean value theorem and the estimate \eqref{portorico3},  we get that

	\begin{align*}
	 ||a_1|^\gamma-|a_2|^\gamma|&\leq \gamma \dfrac{|a_1-a_2|}{|a_1|}|a_1|^\gamma 17^{\gamma-1}\leq  \gamma CL_1\left( L_2 \beta \abs{\bar x-\bar y}^{\beta-1}\right)^{\gamma-1} &\text{if}&\quad \gamma\geq 1\\
	 &\leq |a_1-a_2|^\gamma \leq (4L_1)^\gamma  &\text{if}&\quad 0\leq \gamma\leq 1\\	
	  &\leq C |a_1-a_2|^\kappa(|a_1|^{\gamma-\kappa}+|a_2|^{\gamma-\kappa}) \leq CL_1^\kappa \left( L_2 \beta \abs{\bar x-\bar y}^{\beta-1}\right)^{\gamma-\kappa} &\text{if}&\,-1<\gamma< 0 
	\end{align*}
	where $0<\kappa<1$.\\
	
	\noindent It follows that 
	\begin{align}\label{nousatt}
	   &||a_1|^\gamma -|a_2|^\gamma|| \tr(\A(a_2)Y)|\leq n\left\|Y\right\| \,\left\|\A(a_2)\right\|\,||a_1|^\gamma-|a_2|^\gamma|\nonumber\\
&\begin{array}{lll}
	 &\quad\leq CL_2\beta\abs{\bar x-\bar y}^{\beta-2}(1+|p-2|)L_1 \left( L_2 \beta \abs{\bar x-\bar y}^{\beta-1}\right)^{\gamma-1}  &\text{if}\quad \gamma\geq 1\\
	 &\quad\leq CL_2\beta\abs{\bar x-\bar y}^{\beta-2}(1+|p-2|) L_1^\gamma &\text{if}\quad  \gamma\in [0,1]\\	
	 &\quad\leq  CL_2\beta\abs{\bar x-\bar y}^{\beta-2}(1+|p-2|)L_1^\kappa \left( L_2 \beta \abs{\bar x-\bar y}^{\beta-1}\right)^{\gamma-\kappa}&\text{if}\quad \gamma\in (-1, 0).
	 \end{array}
	 \end{align}
	 In order to estimate ($i_4$),  we use  the estimate \eqref{portorico3}, and get 
		$$ L_1(|a_1|^\gamma\tr(\A(a_1))+|a_2|^\gamma\tr(\A(a_2)))\leq  2L_1n\max(1, p-1)C(\beta)\left(L_2 \abs{\bar x-\bar y}^{\beta-1}\right)^{\gamma}.$$
		Finally, gathering the previous estimates and plugging them into \eqref{gregor1}, we get 
		\begin{align*}
		0&\leq 4L_1 + 2 \norm{f}_{L^\infty(Q_1)}+2L_1n\max(1, p-1)C(\beta)\left(L_2 \abs{\bar x-\bar y}^{\beta-1}\right)^{\gamma} \\
		& \quad+C\min(1, p-1)\left( L_2 \beta \abs{\bar x-\bar y}^{\beta-1}\right)^\gamma L_2\beta\abs{\bar x-\bar y}^{\beta-2}\left(\frac{\beta-1}{3-\beta}\right)\\
		&\quad+C\left( L_2 \beta \abs{\bar x-\bar y}^{\beta-1}\right)^\gamma \abs{p-2}L_1 \abs{\bar x-\bar y}^{-1}\\
		&\quad+ \text{\bf right hand term of}\quad \eqref{nousatt}.
		\end{align*}
		 Choosing $L_2$ large enough  \begin{align}L_2&\geq C(L_1+L_1^{\frac{1}{1+\gamma}}+ \norm{ f}_{L^\infty(Q_1)}^{\frac{1}{1+\gamma}})\geq C(\norm{u}_{L^\infty(Q_1)}+ \norm{u}_{L^\infty(Q_1)}^{\frac{1}{\gamma+1}}+\norm{ f}_{L^\infty(Q_1)}^{\frac{1}{\gamma+1}}),\label{dependfho}
		 \end{align}  
		 we end up with
		$$ 0\leq \dfrac{\min(1, p-1)\beta(\beta-1)}{1000(3-\beta)} L_2\abs{\bar x-\bar y}^{\beta-2}<0,   $$
		which is  a contradiction.   Hence,  $\Phi(x,y,t)\leq 0$ for $(x,y)\in \overline{B_{15/16}}\times\overline{ B_{15/16}}$ and $t\in [-(15/16)^2, 0]$.
		 This concludes the proof  since for $x_0,y_0\in B_{15/16}$ and $t_0\in (-(15/16)^2, 0]$, we have $\Phi(x_0,y_0, t_0)\leq 0$  and we get
		\[
		|u(x_0, t_0)-u(y_0, t_0)|\leq L_2|x_0-y_0|^\beta.
		\]
		Remembering the dependence of $L_2$ (see \eqref{dependfho}), we get the desired result.
\end{proof}

	\subsection{Local Lipschitz estimates}

	\begin{lemma}\label{liphomeq}
	Let $-1<\gamma<\infty$ and $1<p<\infty$.	Let $u$ be a bounded viscosity solution to equation \eqref{maineq}. For all $r\in \left(0,\frac78\right)$, and  for all  $x,y\in \overline{B_{r}}$ and $t\in [-r^2,0]$, it holds
		\begin{equation}
		\begin{split}
		\abs{u(x,t)-u(y,t)}\leq \tilde C \left(\norm{ u}_{L^\infty(Q_1)}+\norm{ u}_{L^\infty(Q_1)}^{\frac{1}{1+\gamma}}+\norm{ f}_{L^\infty(Q_1)}^{\frac{1}{1+\gamma}}\right) \abs{x-y}, 
		\end{split}
		\end{equation}
		where  $\tilde{C}=\tilde{C}(p,n, \gamma)>0$.
		
	\end{lemma}
	\begin{proof}
 In the sequel we fix $r=7/8$ and  we fix $x_0, y_0\in B_{r}$, $t_0\in (-r^2,0)$. For positive constants $L_1, L_2$, we consider the  function 
		\begin{align*}
		\Phi(x, y,t):&=u(x,t)-u(y,t)-L_2\vp(\abs{x-y})-\frac {L_1}{2}\abs{x-x_0}^2-\frac {L_1}{2}\abs{y-y_0}^2-\frac{ L_1}{ 2} (t-t_0)^2,
		\end{align*}
		where $\vp$ is defined below. We want to show that $\Phi(x, y,t)\leq 0$ for $(x,y)\in \overline{B_r}\times\overline{ B_r}$ and $t\in [-r^2,0]$.
	 This time we take
		\[
		\begin{split}
		\vp(s)=
		\begin{cases}
		s-s^{\nu}\kappa_0& 0\le s\le s_1:=(\frac 1 {4\nu\kappa_0})^{1/(\nu-1)}  \\
		\vp(s_1)& \text{otherwise},
		\end{cases}
		\end{split}
		\]
		where $2>\nu>1$ and $\kappa_0>0$ is taken so that  $s_1>2 $ and $\nu \kappa_0s_1^{\nu-1}\leq 1/4$.
		
\noindent Then $\vp$ is smooth in $(0, s_1)$ and for $s\in (0, s_1)$ we have
		$$
		\begin{cases}
		\vp'(s)&=
		1-\nu s^{\nu-1}\kappa_0,\\
		\vp''(s)&=
		-\nu(\nu-1)s^{\nu-2} \kappa_0. 
	\end{cases}$$
		Next, observe that with these choices we have $\varphi'(s)\in  [\frac34,1]$  and $\varphi''(s)<0$ when $s\in (0,2]$.

 We proceed by contradiction assuming that
		$\Phi$ has a positive
		maximum at some point $(\bar x, \bar y,\bar t)\in \bar B_r\times \bar B_r\times [-r^2,0]$ and  we are going to get a contradiction for $L_2$, $L_1$
large enough and for a suitable choice of $\nu$.
	As in the proof of the Hölder estimate, 	we notice  that $\bar x\neq \bar y$ and  for $L_1\geq C\norm{u}_{L^\infty(Q_1)}$,  we have 
that $\bar x$ and $\bar y$ are in $B_r$ and $\bar t\in (-r^2,0)$. Moreover,  from Lemma \ref{holdnormsol}  we know that  $u$ is locally H\"older continuous, and that for any $\beta\in (0,1)$  there exists a constant $C_{H}>0$ 
$$C_{H}:= C\times \left(\norm{ u}_{L^\infty(Q_1)}+\norm{ u}_{L^\infty(Q_1)}^{\frac{1}{1+\gamma}}+\norm{ f}_{L^\infty(Q_1)}^{\frac{1}{1+\gamma}}\right)$$ 
 such that 
		$$|u(x,t)-u(y,t )|\leq C_{H}|x-y|^\beta \quad\text{for}\, x, y\in B_r, t\in (-r^2,0).$$	
		Using this estimate and adjusting the constants (by  choosing $2L_1\leq C_{H}$), we have that  
		\begin{equation}\label{karatte}
L_1\abs{\bar y-y_0},		L_1\abs{\bar x-x_0}\leq C_{H}\abs{\bar x-\bar y}^{\beta/2}.
		\end{equation} 	
From the  Jensen-Ishii's lemma, we have  the existence of 
		
		\[
		\begin{split}
		&(\sigma+L_1(\bar t-t_0),a_1,X+L_1I)\in \ol P^{2,+}u(\bar x,\bar t),\\ &(\sigma,a_2,Y-L_1I)\in \ol P^{2,-}u(\bar y,\bar t),
		\end{split}
		\]
		where 
		\[
		\begin{split}
		a_1&=L_2\vp'(|\bar x-\bar y|) \frac{\bar x-\bar y}{\abs{\bar x-\bar y}}+L_1(\bar x-x_0),\\
		a_2&=L_2\vp'(|\bar x-\bar y|) \frac{\bar x-\bar y}{\abs{\bar x-\bar y}}-L_1(\bar y-y_0).
		\end{split}
		\]
	Recalling that $\vp'\geq \frac34$, then  if  $L_2\geq 4C_{H}$, we have
		\begin{equation}\label{hertzi}
	2L_2\geq	\abs{a_1},\abs{a_2}\geq L_2\varphi'(|\bar x-\bar y|) - C_{H}\abs{\bar x-\bar y}^{\beta/2}\ge \frac{L_2}{2}.
		\end{equation}
		
\noindent Moreover, by  Jensen-Ishii's lemma, for any $\tau>0$, we can take $X, Y\in \mathcal{S}^n$ such that 
		\begin{equation}\label{matineq22}
		- \big[\tau+2\norm{Z}\big] \begin{pmatrix}
		I&0\\
		0&I 
		\end{pmatrix}\leq	\begin{pmatrix}
		X&0\\
		0&-Y 
		\end{pmatrix}
		\le 
		\begin{pmatrix}
		Z&-Z\\
		-Z&Z 
		\end{pmatrix}
		+\frac2\tau \begin{pmatrix}
		Z^2&-Z^2\\
		-Z^2&Z^2 
		\end{pmatrix},
		\end{equation}
		where 
		\begin{align*}	
		Z=L_2\vp''(|\bar x-\bar y|) \frac{\bar x-\bar y}{\abs{\bar x-\bar y}}\otimes \frac{\bar x-\bar y}{\abs{\bar x-\bar y}} +\frac{L_2\vp'(|\bar x-\bar y|)}{\abs{\bar x-\bar y}}\Bigg( I- \frac{\bar x-\bar y}{\abs{\bar x-\bar y}}\otimes \frac{\bar x-\bar y}{\abs{\bar x-\bar y}}\Bigg)
		\end{align*}	
		and 
		\begin{align*}	
		Z^2=
		\frac{L_2^2(\vp'(|\bar x-\bar y|))^2}{\abs{\bar x-\bar y}^2}\Bigg( I- \frac{\bar x-\bar y}{\abs{\bar x-\bar y}}\otimes \frac{\bar x-\bar y}{\abs{\bar x-\bar y}}\Bigg)+L_2^2(\vp''(|\bar x-\bar y|))^2 \frac{\bar x-\bar y}{\abs{\bar x-\bar y}}\otimes \frac{\bar x-\bar y}{\abs{\bar x-\bar y}}.
\end{align*}	
Simple computations give
		\begin{align}\label{lilou}
		\norm{Z}&\leq L_2 \frac{\vp'(|\bar x-\bar y|)}{|\bar x-\bar y|},\\
	 \norm{Z^2}&\leq L_2^2\left(|\vp''(|\bar x-\bar y|)|+\dfrac{|\vp'(|\bar x-\bar y|)|}{|\bar x-\bar y|}\right)^2,\label{filo}
		\end{align}
		and for $\xi=\frac{\bar x-\bar y}{\abs{\bar x-\bar y}}$, we have
	
				\begin{equation*}
		\langle Z\xi,\xi\rangle=L_2\vp''(|\bar x-\bar y|)<0, \qquad\langle Z^2\xi,\xi\rangle=L_2^2(\vp''(|\bar x-\bar y|))^2.
		\end{equation*}
		We take  $\tau=4L_2\left(|\vp''(|\bar x-\bar y|)|+\dfrac{|\vp'(|\bar x-\bar y|)|}{|\bar x-\bar y|}\right)$ and we observe that for $\xi=\frac{\bar x-\bar y}{\abs{\bar x-\bar y}}$,
			\begin{align}\label{mercit3}
		\langle Z\xi,\xi\rangle +\frac2\tau \langle Z^2\xi,\xi\rangle&=L_2\left(\vp''(|\bar x-\bar y|)+\frac2\tau L_2(\vp''(|\bar x-\bar y|))^2\right)\nonumber\\
		&\leq \dfrac{L_2}{2}\vp''(|\bar x-\bar y|)<0 .
		\end{align}
		From  \eqref{matineq22}, we deduce that $X- Y\leq 0$ and $\norm{X},\norm{Y}\leq 2\norm{Z}+\tau$. Moreover,  applying the matrix inequality \eqref{matineq22} to the vector $(\xi,-\xi)$ where $\xi:=\frac{\bar x-\bar y}{|\bar x-\bar y|}$  and using \eqref{mercit3}, 
		we obtain
		\begin{align}\label{camilleprt}
		\langle (X-Y) \xi, \xi\rangle&\leq 4\left(\langle Z\xi,\xi\rangle+\frac2\tau\langle Z^2\xi,\xi\rangle\right)\leq 2 L_2\vp''(|\bar x-\bar y|)<0.
		\end{align}
	This implies that  at least one of the eigenvalue of $X-Y$  that we denote by  $\lambda_{i_0}$ is   negative and smaller than $2 L_2\vp''(|\bar x-\bar y|)$. Writing the  viscosity inequalities and adding them, we have 
			\begin{align}\label{gregory1fout}
		0&\leq 2(L_1+||f||_{L^\infty(Q_1)} )+|a_1|^\gamma\underbrace{\tr (\A(a_1)(X-Y))}_{(1)}	+\underbrace{|a_1|^\gamma \tr ((\A(a_1)-\A(a_2))Y)}_{(2)}\nonumber\\
		&+\underbrace{(|a_1|^\gamma-|a_2|^\gamma)\tr(\A(a_2)Y)}_{(3)}+\underbrace{L_1\big[|a_1|^\gamma\tr (\A(a_1))+|a_2|^\gamma\tr (\A(a_2))}_{(4)} \big].
		\end{align}
		 The eigenvalues of $\A(a_1)$ belong to $[\min(1, p-1), \max(1, p-1)]$.	Using \eqref{camilleprt}, it follows  that we can estimate $(1)$ by
		\begin{align*}  
		\tr(\A(a_1) (X-Y))&\leq \sum_i \lambda_i(\A(a_1))\lambda_i(X-Y)\\
		&\leq \min(1, p-1)\lambda_{i_0}(X-Y)\\
		&\leq 2\min(1, p-1) L_2 \vp''(|\bar x-\bar y|).
		\end{align*}
		As in the proof of the Hölder estimate, we estimate $(2)$ by
		\begin{align*}
		\tr( (\A(a_1)-\A(a_2)) Y)
		\leq 2n\abs{p-2}\norm{Y}|\hat a_1-\hat\eta_2|.
		\end{align*}
	With the new choice of $\varphi$,  we have
		\begin{equation*}
		\begin{split}
		\abs{\hat a_1-\hat a_2}=
		\abs{\frac{a_1}{\abs {a_1}}-\frac{a_2}{\abs {a_2}}}
\le \max\left( \frac{\abs{a_2- a_1}}{\abs{a_2}},\frac{ \abs{a_2- a_1}}{\abs{a_1}}\right)\le \frac {8C_{H}}{ L_2}\abs{\bar x-\bar y}^{\beta/2},
		\end{split}
		\end{equation*}
		where we used \eqref{karatte} and \eqref{hertzi}.
		Using  \eqref{matineq22}--\eqref{filo}, we have
		\begin{equation*}
		\norm{Y}
		\le 2 |\langle Z\ol \xi,\ol \xi \rangle|+\frac4\tau|\langle Z^2\ol \xi,\ol \xi \rangle| \leq 4L_2\left( \frac{\vp'(|\bar x-\bar y|)}{\abs{\bar x-\bar y}}+ |\vp''(|\bar x-\bar y|)|\right).
		\end{equation*}
		Hence, remembering that $|\bar x-\bar y|\leq 2$ and $|a_1|^\gamma \leq CL_2^\gamma$, we end up with
	\begin{align*}
	|a_1|^\gamma| \tr( (\A(a_1)-\A(a_2)) Y)|&\leq CL_2^\gamma \abs{p-2}C_{H} \vp'(|\bar x-\bar y|) \abs{\bar x-\bar y}^{-1+\beta/2}\\
	 &\quad+C L_2^\gamma \abs{p-2}C_{H} |\vp''(|\bar x-\bar y|)|.
	 \end{align*}
	 Using the mean value theorem and the estimates \eqref{karatte} and \eqref{hertzi}, we have
	 \begin{align*}
	 ||a_1|^\gamma-|a_2|^\gamma|&\leq \gamma \dfrac{|a_1-a_2|}{|a_1|}|a_1|^\gamma 8^{\gamma-1}\leq  CC_{H} L_2^{\gamma-1} \abs{\bar x-\bar y}^{\beta/2}&\text{if}&\quad \gamma\geq 1\\
	 &\leq |a_1-a_2|^\gamma \leq (C_{H} |\bar x-\bar y|^{\beta/2})^\gamma  &\text{if}&\quad 0< \gamma\leq 1\\	
	  &\leq |a_1-a_2|^\kappa(|a_1|^{\gamma-\kappa}+|a_2|^{\gamma-\kappa})\leq CL_2^{\gamma-\kappa} \left( C_{H} \abs{\bar x-\bar y}^{\beta/2}\right)^{\kappa}  &\text{if}&\, -1<\gamma\leq  0
	\end{align*}
	where $0<\kappa<1$.\\
	It follows (using that $|\bar x-\bar y|)\leq 2$)
\begin{align}\label{nousa1li}
||a_1|^\gamma-&|a_2|^\gamma||\tr(\A(a_2)Y)|\leq n\left\|Y\right\| \,\left\|\A(a_2)\right\|\, ||a_1|^\gamma-|a_2|^\gamma|\\
&\begin{array}{lll}
	 \leq  CL_2\left( \abs{\bar x-\bar y}^{-1+\beta/2}+ |\vp''(|\bar x-\bar y|)|\right)(1+|p-2|)C_{H} L_2^{\gamma-1}   &\text{if}\quad \gamma\geq 1\\
	\leq  CL_2\left( \frac{\vp'(|\bar x-\bar y|)}{\abs{\bar x-\bar y}}+ |\vp''(|\bar x-\bar y|)|\right)(1+|p-2|)
	 C_{H}^\gamma|\bar x-\bar y|^{\frac{\gamma\beta}{2}} &\text{if}\quad 0<\gamma\leq 1\\	
	 \leq  CL_2^{1+\gamma-\kappa}\left( \frac{\vp'(|\bar x-\bar y|)}{\abs{\bar x-\bar y}}+ |\vp''(|\bar x-\bar y|)|\right)(1+|p-2|)|\bar x-\bar y|^{\frac{\kappa\beta}{2}}C_{H}^\kappa  &\text{if}\, -1<\gamma\leq 0.\nonumber
	\end{array}
	 \end{align}
		 Finally, we have
		$$ L_1[|a_1|^\gamma\tr(\A(a_1))+|a_2|^\gamma\tr(\A(a_2))]\leq 2CL_2^\gamma L_1n\max(1, p-1).$$
		Gathering the previous estimates with \eqref{gregory1fout} and recalling the definition of $\vp$, we get 

	\begin{align*}
		0&\leq 2(L_1+\norm{ f}_{L^\infty(Q_1)})+CL_2^\gamma L_1\max(1, p-1)\\
		&\quad +CL_2^\gamma \abs{p-2}C_{H} \abs{\bar x-\bar y}^{-1+\beta/2}+C L_2^\gamma \abs{p-2}C_{H} |\bar x-\bar y|^{\nu-2}\\
		&\quad - 2CL_2^\gamma \min(1, p-1) L_2(\nu-1)\nu \kappa_0|\bar x-\bar y|^{\nu-2}+\text{\bf right hand term of}\quad \eqref{nousa1li}.
		\end{align*}
		Taking 
		$\nu=1+\frac\beta 2$, recalling the dependence of $\bar C_H$ and choosing $L_2$ large 
$$L_2\geq  C \left(\norm{ u}_{L^\infty(Q_1)}+\norm{ u}_{L^\infty(Q_1)}^{\frac{1}{1+\gamma}}+\norm{ f}_{L^\infty(Q_1)}^{\frac{1}{1+\gamma}}\right),	$$ 
	  we get that
		$$ 0\leq \dfrac{-\min(1, p-1)\nu(\nu-1)\kappa_0}{1000} L_2\abs{\bar x-\bar y}^{\nu-2}<0,   $$
		which is  a contradiction. 
		  It follows that  $\Phi(x,y,t)\leq 0$ for $(x,y,t)\in \overline{B_r\times} \overline{B_{r}}\times[-r^2,0]$. The desired result follows since for $x_0,y_0\in B_{r}$, $t_0\in (-r^2,0)$, we have $\Phi(x_0,y_0,t_0)\leq 0$, so that 
		\[
		|u(x_0,t_0)-u(y_0,t_0)|\leq L_2\vp(|x_0-y_0|)\leq L_2|x_0-y_0|.\qedhere
		\]
		
	\end{proof}
\section{Proof of the uniform  Hölder and Lipschitz estimates}\label{sect7}
In this section  we provide a proof for Lemma \ref{Liphom1}.   Assume that  $0\leq\gamma<\infty$ and consider  bounded solutions  $w$ to 
\begin{equation}\label{deviaapen}
	\partial_t w-|Dw+q|^\gamma\left[\Delta w+(p-2) \left\langle D^2w\frac{Dw+q}{\abs{Dw+q}}, \frac{Dw+q}{\abs{Dw+q}}\right\rangle\right]= \bar f\quad\text{in}\quad Q_1.
	\end{equation}
Noticing that $h(x,t):= w(x,t)+q\cdot x$  is a solution of \eqref{maineq}. It follows from Lemma \ref{liphomeq} that $w$ is Lipschitz continuous with respect to the space variable. Moreover, for  $x,y\in B_{7/8}$ and $t\in(-(7/8)^2, 0]$, it holds
\begin{align}\label{girou}
|w(x,t)-w(y,t)|&\leq |h(x, t)-h(y,t)|+|q||x-y|\nonumber\\
&\leq \left(|q|+C(\norm{h}_{L^\infty(Q_1)}+\norm{f}_{L^\infty(Q_1)}^{\frac{1}{\gamma+1}}+\norm{h}_{L^\infty(Q_1)}^{\frac{1}{\gamma+1}})\right)|x-y|\nonumber\\
&\leq C\left(|q|+1+\norm{w}_{L^\infty(Q_1)}+\norm{ f}_{L^\infty(Q_1)}\right)|x-y|.
\end{align}
Hence, 
if $|q|\geq\Gamma_0:= 2+\norm{w}_{L^\infty(Q_1)}+\norm{\bar f}_{L^\infty(Q_1)}$, then  for $(x,t), (y,t) \in Q_{7/8}$, we have
\begin{equation}\label{ramisto}|w(x,t)-w(y,t)|\leq C(p,n,\gamma)|q||x-y|.
\end{equation}
We will improve this estimate and  provide uniform estimates for deviation from planes with $|q|>\Gamma_0$. The case  $|q|<\Gamma_0$ follows from  \eqref{girou}.
In order to prove uniform Lipschitz estimates with respect to $q$, we first need to prove uniform Hölder estimates.
	\subsection{Local uniform Hölder estimates}
	\begin{lemma}\label{holdnorm}
		Let $w$ be a bounded viscosity solution to equation \eqref{deviaapen} with  $|q|>\Gamma_0$.  There exist  a constant  $\beta=\beta(p,n, \gamma)\in (0,1)$ and  a constant $C=C(p,n,\gamma)>0$ such that   for all  $x,y\in B_{13/16}$ and $t \in (-(13/16)^2,0]$, it holds
		\begin{equation}
		\begin{split}
		\abs{w(x,t)-w(y,t)}\leq C(1+\norm{w}_{L^\infty(Q_1)}+\norm{\bar f}_{L^\infty(Q_1)}) \abs{x-y}^\beta.
		\end{split}
		\end{equation}
		
	\end{lemma}
	\begin{proof}

We fix $x_0, y_0\in B_{13/16}$, $t_0\in (-(13/16)^2,0)$. For suitable  constants $L_1, L_2>0$, we define the auxiliary function 
		\begin{align*}
		\Phi(x, y,t):&=w(x,t)-w(y,t)-L_2\abs{x-y}^\beta-\frac{L_1}{2}\abs{x-x_0}^2-
		\frac{L_1}{2}\abs{y-y_0}^2-\frac{L_1}{ 2} (t-t_0)^2.
		\end{align*}
		 We want to show that $\Phi(x, y,t)\leq 0$ for $(x,y)\in \overline{B_{13/16}}\times\overline{ B_{13/16}}$ and $t\in [-(13/16)^2,0]$.
  We argue by contradiction.  We assume that
		$\Phi$ has a positive
		maximum at some point $(\bar x, \bar y,\bar t)\in \bar B_{13/16}\times \bar B_{13/16}\times [-(13/16)^2,0]$ and we are going to get a contradiction for $L_2$, $L_1$
large enough.
The positivity of the maximum of $\Phi$  implies that $\bar x\neq \bar y$. Choosing $L_1\geq C\norm{w}_{L^\infty(Q_1)}$,
  we have that 
		 $\bar x$ and $\bar y$ are in $B_{13/16}$ and $\bar t\in (-(13/16)^2,0)$. 
We proceed as in the proof of Lemma \ref{holdnormsol}.
		By  Jensen-Ishii's lemma,  there exist
		\[
		\begin{split}
		&(\sigma+L_1(\bar t-t_0), a_1,X+L_1I)\in \overline{\mathcal{P}}^{2,+} w(\bar x, \bar t),\\ &(\sigma, a_2,Y-L_1I)\in \overline{\mathcal{P}}^{2,-} w(\bar y, \bar t).
		\end{split}
		\]
		
		Assuming that  $L_2\geq CL_1$, we have
		\begin{equation}\label{portorico}
		\begin{cases}
		2L_2\beta \abs{\bar x-\bar y}^{\beta-1}&\geq\abs{a_1}\geq L_2\beta|\bar x-\bar y|^{\beta-1} - L_1\abs{\bar x-x_0}\ge \frac{L_2}{2} \beta \abs{\bar x-\bar y}^{\beta-1}\\
		2L_2\beta\abs{\bar x-\bar y}^{\beta-1}&\geq\abs{a_2}\geq L_2\beta|\bar x-\bar y|^{\beta-1} - L_1\abs{\bar y-y_0}\ge \frac{L_2}{2} \beta \abs{\bar x-\bar y}^{\beta-1}.
		\end{cases}
		\end{equation}
	We can take $X, Y\in \mathcal{S}^n$ such that  it holds\\
		\begin{equation}\label{maineq1}
		-\frac{2}{\tau} \begin{pmatrix}
		I&0\\
		0&I 
		\end{pmatrix}\leq
		\begin{pmatrix}
		X&0\\
		0&-Y 
		\end{pmatrix}\leq \begin{pmatrix} Z^\tau& -Z^\tau\\
		-Z^\tau& Z^\tau\end{pmatrix},
		\end{equation}
	
		where 
	\begin{align*}
		Z^\tau=(I-\tau Z)^{-1} Z=2L_2\beta\abs{\bar x-\bar y}^{\beta-2}\left(I-2\frac{2-\beta}{3-\beta} \frac{\bar x-\bar y}{\abs{\bar x-\bar y}}\otimes \frac{\bar x-\bar y}{\abs{\bar x-\bar y}}\right).
\end{align*}	
 We have  for $\xi=\frac{\bar x-\bar y}{\abs{\bar x-\bar y}}$, 
	
				\begin{equation*}
		\langle Z^\tau \xi,\xi\rangle= 2L_2\beta\abs{\bar x-\bar y}^{\beta-2}\left(\frac{\beta-1}{3-\beta}\right)<0.
		\end{equation*}
Applying the inequality \eqref{maineq1} to any  vector $(\xi,\xi)$ with $\abs{\xi}=1$, we get that $X- Y\leq 0$ and 
		\begin{equation*}
		\norm{X},\norm{Y}\leq 4L_2\beta\abs{\bar x-\bar y}^{\beta-2}.
		\end{equation*}
		Moreover, using the positivity of the maximum of $\Phi$ and the Lipschitz regularity of $w$ (see \eqref{ramisto}), we have for $0<\beta\leq \frac{1}{4C}$,
		\begin{equation}\label{pogba}2\beta L_2 \abs{\bar x-\bar y}^{\beta-1}\leq 2\beta \frac{|w(\bar x, \bar t)-w(\bar y, \bar t)|}{|\bar x-\bar y|}\leq 2\beta C(p,n, \gamma)|q|\leq\frac{ |q|}{2}.
		\end{equation}
	Setting $\eta_1=a_1+q$, $\eta_2=a_2+q$, we get by using \eqref{portorico} and \eqref{pogba}, that   
		\begin{align}\label{koivu1}
		2|q|\geq\abs{\eta_1}&\geq \abs{q}-|a_1|\geq \frac{\abs{q}}{2}\geq 2L_2 \beta \abs{\bar x-\bar y}^{\beta-1},\nonumber\\
			2|q|\geq\abs{\eta_2}&\geq \abs{q}-|a_2|\geq \frac{\abs{q}}{2}\geq 2L_2 \beta \abs{\bar x-\bar y}^{\beta-1}.
		\end{align}
Writing the viscosity inequalities and adding them, we  get 

		\begin{align}\label{gregory1}
		0&\leq 2|\eta_1|^{-\gamma}(L_1+||\bar f||_{L^\infty(Q_1)} )+\underbrace{\tr (\A(\eta_1)(X-Y))}_{(i_1)}	+\underbrace{\tr ((\A(\eta_1)-\A(\eta_2))Y)}_{(i_2)\nonumber}\\
		&+\underbrace{ |\eta_1|^{-\gamma}(|\eta_1|^\gamma-|\eta_2|^\gamma)\tr(\A(\eta_2)Y)}_{(i_3)}
		+\underbrace{L_1\big[\tr (\A(\eta_1))+|\eta_2|^\gamma|\eta_1|^{-\gamma}\tr (\A(\eta_2))}_{(i_4)} \big].
		\end{align}
We estimate ($i_1$) as in the proof of Lemma \ref{holdnormsol}
		\begin{align*} \tr(\A(\eta_1) (X-Y))&\leq\min(1, p-1)8L_2\beta\abs{\bar x-\bar y}^{\beta-2}\left(\frac{\beta-1}{3-\beta}\right).
		\end{align*}
In order to estimate  ($i_2$),  we use  that $|\eta_1-\eta_2|\leq 4 L_1$ and  the estimate \eqref{koivu1},  so that 
		\begin{equation*}
		\abs{\hat \eta_1-\hat \eta_2}=
		\abs{\frac{\eta_1}{\abs {\eta_1}}-\frac{\eta_2}{\abs {\eta_2}}}
\le \max\left( \frac{\abs{\eta_2- \eta_1}}{\abs{\eta_2}},\frac{ \abs{\eta_2- \eta_1}}{\abs{\eta_1}}\right)\le \frac {16  L_1}{ \beta L_2 \abs{\bar x-\bar y}^{\beta-1}}.
		\end{equation*}
	Recalling that  $
		\norm{Y}=\underset{\hat \xi}{\max}\, |\langle Y\hat \xi, \hat \xi\rangle|
		\le  4L_2\beta\abs{\bar x-\bar y}^{\beta-2}$, it follows that 
	\begin{align*}
	 \tr( (\A(\eta_1)-\A(\eta_2)) Y)|\leq Cn\abs{p-2}L_1 \abs{\bar x-\bar y}^{-1}.
	 \end{align*}
	 Now we estimate the term ($i_3$). Notice that $|\eta_2|/|\eta_1|\leq 16$, $|\eta_1|\geq 1$ and $|\eta_1-\eta_2|\leq 4L_1$.  Using  the mean value theorem and the estimate \eqref{koivu1},  we get that  
	\begin{align*}
	 ||\eta_1|^\gamma-|\eta_2|^\gamma|&\leq C\gamma \dfrac{|\eta_1-\eta_2|}{|\eta_1|}|\eta_1|^\gamma\leq  \gamma CL_1|\eta_1|^{\gamma-1} &\text{if}&\quad \gamma\geq 1\\
	 &\leq |\eta_1-\eta_2|^\gamma \leq (4L_1)^\gamma  &\text{if}&\quad 0\leq \gamma\leq 1.
	\end{align*}
	\noindent It follows that 
	\begin{align}\label{nousa}
	 \dfrac{||\eta_1|^\gamma -|\eta_2|^\gamma||\tr(\A(\eta_2)Y)|}{  |\eta_1|^{\gamma}} \leq& n|\eta_1|^{-\gamma}\left\|Y\right\| \,\left\|\A(\eta_2)\right\|\,||\eta_1|^\gamma-|\eta_2|^\gamma|&&\nonumber\\
	 \leq& C\abs{\bar x-\bar y}^{-1}(1+|p-2|)L_1   &\text{if}&\quad \gamma\geq 1\nonumber\\
	 \leq& C L_2^{1-\gamma}\beta^{1-\gamma}\abs{\bar x-\bar y}^{\beta-2+\gamma(1-\beta)}(1+|p-2|) L_1^\gamma &\text{if}&\quad\gamma\in [0,1].
	 \end{align}
	We estimate ($i_4$) by
		$$ L_1(\tr(\A(\eta_1))+|\eta_2|^\gamma|\eta_1|^{-\gamma}\tr(\A(\eta_2)))\leq  CL_1n\max(1, p-1).$$
		Finally, gathering the previous estimates and plugging them into \eqref{gregory1} and recalling that $|\eta_1|\geq 1$, we get 
		\begin{align*}
		0&\leq 4L_1 + 2 \norm{\bar f}_{L^\infty(Q_1)}+CL_1n\max(1, p-1) +C\min(1, p-1) L_2\beta\abs{\bar x-\bar y}^{\beta-2}\left(\frac{\beta-1}{3-\beta}\right)\\
		&\quad+C n\abs{p-2}L_1 \abs{\bar x-\bar y}^{-1}+ \text{\bf right hand term of}\quad \eqref{nousa}.
		\end{align*}
		 Choosing $L_2$ large enough  \begin{align}L_2\geq C(1+L_1+ \norm{\bar f}_{L^\infty(Q_1)} )\geq C(1+\norm{w}_{L^\infty(Q_1)}+\norm{\bar f}_{L^\infty(Q_1)} ),\label{dependf}
		 \end{align}  we end up with
		$$ 0\leq \dfrac{\min(1, p-1)\beta(\beta-1)}{1000(3-\beta)} L_2\abs{\bar x-\bar y}^{\beta-2}<0,   $$
		which is  a contradiction.  
		 This concludes the proof  since for $(x_0,t_0), (y_0,t_0)\in Q_{13/16}$, we have $\Phi(x_0,y_0, t_0)\leq 0$ and  we get
		\[
		|u(x_0, t_0)-u(y_0, t_0)|\leq L_2|x_0-y_0|^\beta.
		\]
		Remembering the dependence of $L_2$ (see \eqref{dependf}), we get the desired result.
\end{proof}

	\subsection{Local uniform Lipschitz estimates}

	\begin{lemma}\label{liphomol}
		Let $w$ be a bounded viscosity solution to equation \eqref{deviaapen} with  $|q|>C\Gamma_0$. For all $r\in \left(0,\frac34\right)$, and  for all  $x,y\in \overline{B_{r}}$ and $t\in [-r^2,0]$, it holds
		\begin{equation}
		\begin{split}
		\abs{w(x,t)-w(y,t)}\leq \tilde C \left(1+\norm{ w}_{L^\infty(Q_1)}+\norm{\bar f}_{L^\infty(Q_1)}\right) \abs{x-y}, 
		\end{split}
		\end{equation}
		where  $\tilde{C}=\tilde{C}(p,n, \gamma)>0$.
		
	\end{lemma}
	\begin{proof}
	We proceed as in the proof of Lemma \ref{liphomeq}. 
We fix $r=3/4$ and  we fix $x_0, y_0\in B_{r}$, $t_0\in (-r^2,0)$. For positive constants $L_1, L_2$, we  define  the  function 
		\begin{align*}
		\Phi(x, y,t):&=w(x,t)-w(y,t)-L_2\vp(\abs{x-y})-\frac {L_1}{2}\abs{x-x_0}^2-\frac {L_1}{2}\abs{y-y_0}^2-\frac{ L_1}{ 2} (t-t_0)^2,
		\end{align*}
		where 
		\[
		\begin{split}
		\vp(s)=
		\begin{cases}
		s-s^{\nu}\kappa_0& 0\le s\le s_1:=(\frac 1 {\nu\kappa_0})^{1/(\nu-1)}  \\
		\vp(s_1)& \text{otherwise},
		\end{cases}
		\end{split}
		\]
		where $2>\nu>1$ and $\kappa_0>0$ is taken so that  $s_1>2 $ and $\nu \kappa_0s_1^{\nu-1}\leq 1/4$.
		We want to show that $\Phi(x, y,t)\leq 0$ for $(x,y)\in \overline{B_r}\times\overline{ B_r}$ and $t\in [-r^2,0]$.  We proceed by contradiction assuming that
		$\Phi$ has a positive
		maximum at some point $(\bar x, \bar y,\bar t)\in \bar B_r\times \bar B_r\times [-r^2,0]$ and  we are going to get a contradiction for $L_2$, $L_1$
large enough and for a suitable choice of $\nu$.
	As in the proof of the Hölder estimate, 	we notice  that $\bar x\neq \bar y$, and  for $L_1\geq C\norm{w}_{L^\infty(Q_1)}$, we have 
that $\bar x$ and $\bar y$ are in $B_r$ and $\bar t\in (-r^2,0)$. Moreover,  from Lemma \ref{holdnorm},  we know that  $w$ is locally H\"older continuous, and that  there exist a constant $C_{H}>0$ 
$$\bar C_{H}:= \bar C\times \left(1+\norm{ u}_{L^\infty(Q_1)}+\norm{\bar f}_{L^\infty(Q_1)}\right)$$
and a constant $\beta=\beta(p,n, \gamma)\in (0,1)$
 such that 
		$$|w(x,t)-w(y,t )|\leq \bar C_{H}|x-y|^\beta \quad\text{for}\, x, y\in B_r, t\in (-r^2,0).$$	
		Using this estimate, we have that  
		\begin{equation*}
	L_1\abs{\bar y-y_0}, 	\,L_1\abs{\bar x-x_0}\leq \bar C_{H}\abs{\bar x-\bar y}^{\beta/2}.
		\end{equation*} 	
From the  Jensen-Ishii's lemma, we have  the existence of 	

		\[
		\begin{split}
		&(\sigma+L_1(\bar t-t_0),a_1,X+L_1I)\in \ol P^{2,+}u(\bar x,\bar t),\\ &(\sigma,a_2,Y-L_1I)\in \ol P^{2,-}u(\bar y,\bar t),
		\end{split}
		\]
	where
		\[
		\begin{split}
		a_1&=L_2\vp'(|\bar x-\bar y|) \frac{\bar x-\bar y}{\abs{\bar x-\bar y}}+L_1(\bar x-x_0),\\
		a_2&=L_2\vp'(|\bar x-\bar y|) \frac{\bar x-\bar y}{\abs{\bar x-\bar y}}-L_1(\bar y-y_0).
		\end{split}
		\]
	Recalling that $\vp'\geq \frac34$, then  if  $L_2$ is large ($L_2\geq 4\bar C_{H}$), we have
		\[
	2L_2\geq	\abs{a_1},\abs{a_2}\geq L_2\varphi'(|\bar x-\bar y|) -\bar C_{H}\abs{\bar x-\bar y}^{\beta/2}\ge \frac{L_2}{2}.
		\]
		Denoting  $\eta_1=a_1+q$, $\eta_2=a_2+q$, we have for $|q|\geq L_2+2$
		\begin{align}\label{koivu}
		3|q|\geq \abs{\eta_1}&\geq \abs{q}-\abs{a_1}\geq \frac{\abs{q}}{2}\geq \frac{L_2}{2},\nonumber\\
		3|q|\geq \abs{\eta_2}&\geq \abs{q}-\abs{a_2}\geq \frac{\abs{q}}{2}\geq \frac{L_2}{2}\\
		|\eta_1-\eta_2|&\leq 2\bar C_H |\bar x-\bar y|^{\beta/2}.\nonumber
		\end{align}
\noindent Moreover, by  Jensen-Ishii's lemma, for any $\tau>0$, we can take $X, Y\in \mathcal{S}^n$ such that 
		
		\begin{equation}\label{matineq2}
		- \big[\tau+2\norm{Z}\big] \begin{pmatrix}
		I&0\\
		0&I 
		\end{pmatrix}\leq	\begin{pmatrix}
		X&0\\
		0&-Y 
		\end{pmatrix}
		\le 
		\begin{pmatrix}
		Z&-Z\\
		-Z&Z 
		\end{pmatrix}
		+\frac2\tau \begin{pmatrix}
		Z^2&-Z^2\\
		-Z^2&Z^2 
		\end{pmatrix}.
		\end{equation}
		
	As in the proof of Lemma \ref{liphomeq}, we take  $\tau=4L_2\left(|\vp''(|\bar x-\bar y|)|+\dfrac{|\vp'(|\bar x-\bar y|)|}{|\bar x-\bar y|}\right)$ and we observe that for $\xi=\frac{\bar x-\bar y}{\abs{\bar x-\bar y}}$,
	
			\begin{align}\label{mercit}
		\langle Z\xi,\xi\rangle +\frac2\tau \langle Z^2\xi,\xi\rangle\leq \dfrac{L_2}{2}\vp''(|\bar x-\bar y|)<0 .
		\end{align}
		We also have that  $X- Y\leq 0$,  $\norm{X},\norm{Y}\leq 2\norm{Z}+\tau$, and that   at least one of the eigenvalue of $X-Y$  that we denote by  $\lambda_{i_0}$ is   negative and smaller than $2 L_2\vp''(|\bar x-\bar y|)$.\\
		Writing the  viscosity inequalities and adding them, we have 
		
			\begin{align}\label{gregory1f}
		0&\leq 2(L_1+||\bar f||_{L^\infty(Q_1)})|\eta_1|^{-\gamma} +\underbrace{\tr (\A(\eta_1)(X-Y))}_{(1)}+\underbrace{\tr ((\A(\eta_1)-\A(\eta_2))Y)}_{(2)}\nonumber\\
		&+\underbrace{|\eta_1|^{-\gamma}(|\eta_1|^\gamma-|\eta_2|^\gamma)\tr(\A(\eta_2)Y)}_{(3)}+\underbrace{L_1\big[\tr (\A(\eta_1))+|\eta_2|^\gamma|\eta_1|^{-\gamma}\tr (\A(\eta_2))}_{(4)} \big].
		\end{align}
		We estimate $(1)$ by
		\begin{align*}  
		\tr(\A(\eta_1) (X-Y))&\leq \sum_i \lambda_i(\A(\eta_1))\lambda_i(X-Y)
		\leq 2\min(1, p-1) L_2 \vp''(|\bar x-\bar y|).
		\end{align*}
		As previously, we estimate $(2)$ by
	\begin{align*}
| \tr( (\A(\eta_1)-\A(\eta_2)) Y)|&\leq C \abs{p-2}\bar C_{H} \vp'(|\bar x-\bar y|) \abs{\bar x-\bar y}^{-1+\beta/2}+C \abs{p-2}\bar C_{H} |\vp''(|\bar x-\bar y|)|.
	 \end{align*}
We have also
\begin{align}\label{nousa1}
|\eta_1|^{-\gamma}||\eta_1|^\gamma -&|\eta_2|^\gamma|| \tr(\A(\eta_2)Y)|\leq n\left\|Y\right\| \,\left\|\A(\eta_2)\right\|\, ||\eta_1|^\gamma-|\eta_2|^\gamma|\nonumber\\
&\begin{array}{lll}
	 &\leq  C\left( \abs{\bar x-\bar y}^{-1+\beta/2}+ |\vp''(|\bar x-\bar y|)|\right)(1+|p-2|)\bar C_{H}  &\text{if}\quad \gamma\geq 1\\
	 &\leq  CL_2^{1-\gamma}\left( \frac{\vp'(|\bar x-\bar y|)}{\abs{\bar x-\bar y}}+ |\vp''(|\bar x-\bar y|)|\right)(1+|p-2|)
	 \bar C_{H}^\gamma|\bar x-\bar y|^{\frac{\gamma\beta}{2}} &\text{if}\quad 0<\gamma\leq 1.
	 \end{array}
	 \end{align}
		 Finally, we have
		$$ |\eta_1|^{-\gamma}L_1[|\eta_1|^\gamma\tr(\A(\eta_1))+|\eta_2|^\gamma\tr(\A(\eta_2))]\leq 2CL_1n\max(1, p-1).$$
		Gathering the previous estimates with \eqref{gregory1f}, recalling the definition of $\vp$ and using that $|\eta_1|\geq |q|/2\geq 1$,  we get 
	\begin{align*}
		0&\leq C(L_1+\norm{\bar f}_{L^\infty(Q_1)})+C L_1n\max(1, p-1)+C\abs{p-2}\bar C_{H} \abs{\bar x-\bar y}^{-1+\beta/2}\\
		&\quad+C \abs{p-2}\bar C_{H} |\bar x-\bar y|^{\nu-2} - 2 \min(1, p-1) (\nu-1)\nu \kappa_0L_2|\bar x-\bar y|^{\nu-2}\\
		&\quad+\text{\bf right hand term of}\quad \eqref{nousa1}.
		\end{align*}
		Taking 
		$\nu=1+\frac\beta 2$, recalling the dependence of $C_H$ and choosing $L_2$ large 
$$L_2\geq  C \left(1+\norm{ u}_{L^\infty(Q_1)}+\norm{\bar f}_{L^\infty(Q_1)}\right),	$$ 
	  we get that
	  
		$$ 0\leq \dfrac{-\min(1, p-1)\nu(\nu-1)\kappa_0}{1000} L_2\abs{\bar x-\bar y}^{\nu-2}<0,   $$
		which is  a contradiction. Hence
		\[
		|u(x_0,t_0)-u(y_0,t_0)|\leq L_2\vp(|x_0-y_0|)\leq L_2|x_0-y_0|.\qedhere
		\]
		
	\end{proof}

\def\cprime{$'$} \def\polhk#1{\setbox0=\hbox{#1}{\ooalign{\hidewidth
  \lower1.5ex\hbox{`}\hidewidth\crcr\unhbox0}}}

	\end{document}